\documentclass[12pt, reqno]{amsart}
\usepackage {amsmath,amssymb,verbatim}
\usepackage[pagebackref,pdftex,hyperindex]{hyperref}

\usepackage[margin=3cm]{geometry}

\newcommand{\A}{\mathbb{A}}
\newcommand{\B}{\mathbb{B}}
\newcommand{\C}{\mathbb{C}}

\newcommand{\G}{\mathbb{G}}
 
\newcommand{\N}{\mathbb{N}}

\renewcommand{\P}{\mathbb{P}}
\newcommand{\Q}{\mathbb{Q}}
\newcommand{\R}{\mathbb{R}}
\renewcommand{\S}{\mathbb{S}}

\newcommand{\Z}{\mathbb{Z}}

\newcommand{\fk}{\mathfrak{k}}

\newcommand{\fs}{\mathfrak{s}}
\newcommand{\ft}{\mathfrak{t}}

\newcommand{\cB}{\mathcal{B}}

\newcommand{\cD}{\mathcal{D}}
\newcommand{\cE}{\mathcal{E}}

\newcommand{\cH}{\mathcal{H}}

\newcommand{\cJ}{\mathcal{J}}

\newcommand{\cL}{\mathcal{L}}

\newcommand{\cO}{\mathcal{O}}

\newcommand{\cR}{\mathcal{R}}

\newcommand{\cX}{\mathcal{X}}

\renewcommand{\a}{\alpha}
\renewcommand{\b}{\beta}

\renewcommand{\d}{\delta}
\newcommand{\e}{\varepsilon}
\newcommand{\s}{\sigma}

\renewcommand{\phi}{\varphi}

\newcommand{\eg}{{\rm e.g.\ }} 
\newcommand{\ie}{{\rm i.e.\ }}

\renewcommand{\leq}{\leqslant}
\renewcommand{\geq}{\geqslant}
\newcommand{\abs}[1]{\left\lvert#1\right\rvert}
\newcommand{\norm}[1]{\left\|#1\right\|}

\renewcommand{\DH}{\mathrm{DH}}

\newcommand{\NA}{\mathrm{NA}}

\DeclareMathOperator{\Aut}{Aut}

\DeclareMathOperator{\Hom}{Hom}
\DeclareMathOperator{\id}{id}
\DeclareMathOperator{\IM}{Im}

\DeclareMathOperator{\lct}{lct}

\DeclareMathOperator{\MA}{MA}
\DeclareMathOperator{\ord}{ord}

\DeclareMathOperator{\PSH}{PSH}
\DeclareMathOperator{\RE}{Re}
\DeclareMathOperator{\Ric}{Ric}

\DeclareMathOperator{\supp}{supp}

\DeclareMathOperator{\Wedge}{\bigwedge}

\numberwithin{equation}{section}       

\newtheorem{prop} {Proposition} [section]
\newtheorem{thm}[prop] {Theorem} 
\newtheorem{dfn}[prop] {Definition}
\newtheorem{lem}[prop] {Lemma}
\newtheorem{cor}[prop]{Corollary}

\newtheorem{exam}[prop]{Example}
\newtheorem{rem}[prop]{Remark}
\theoremstyle{remark}
\newtheorem*{ackn}{\bf{Acknowledgment}}

\newtheorem*{thmA}{\bf{Theorem A}} 
\newtheorem*{thmB}{\bf{Theorem B}}

\newtheorem*{dfn*}{\bf{Definition}}

\title[]{Mabuchi's soliton metric and relative D-stability} 
\date{\today} 

\author{Tomoyuki Hisamoto}
\address{Graduate School of Mathematics\\
  Nagoya University\\
  Furocho\\
  Chikusa\\
  Nagoya\\ 
  Japan}
\email{hisamoto@math.nagoya-u.ac.jp}

\begin{document}

\maketitle

\setcounter{tocdepth}{1}

\begin{abstract}
For Fano manifolds T. Mabuchi introduced a generalization of the K\"ahler-Einstein metric, which is characterized as the critical point of the Ricci-Calabi functional. 
We show that a Fano manifold admits Mabuchi's metric if and only if it is uniformly relatively D-stable. 
The idea of the proof includes some equivariant generalization of the recent developed variational approach to the K\"ahler-Einstein problem. 
\end{abstract}

\tableofcontents 

\section{Introduction}

Let $X$ be a Fano manifold. 
In a central problem of complex geometry we are guided to look for a standard K\"ahler metric in the first Chern class $c_1(X)=c_1(-K_X)$. 
The fundamental result established in \cite{CDS15} states that there exists a K\"ahler-Einstein metric if and only if $X$ is K-polystable (see also \cite{Tian15}). 
Not all the Fano manifold satisfy the stability; for example one-point blow up of $\P^2$ is never K\"ahler-Einstein. 
On the other hand, for an {\em arbitrary} Fano manifold $X$ we may consider a canonical geometric flow which should optimally destabilize $X$. 
The self-similar solution of the flow coincides with T. Mabuchi's generalization of K\"ahler-Einstein metric. 
The purpose of this paper is to clarify which Fano manifold admits such a metric. 

For the definition, let us denote a K\"ahler metric by $\omega$ and the normalized Ricci potential function by $\rho$ which is the unique function satisfies 
\begin{equation}\label{Ricci potential}
\Ric \omega -\omega = dd^c\rho, ~~~
\int_X (e^\rho -1) \omega^n=0. 
\end{equation}
We also write $\omega = dd^c\phi$ locally so as to identify the metric with a collection of smooth functions $\phi$ patching together to define the fiber metric of $-K_X$.  
Our standard metric first introduced by \cite{Mab01} is the critical point of the Ricci-Calabi functional 
\begin{equation}\label{Ricci-Calabi functional}
R(\omega)=R(\phi) := \frac{1}{V}\int_X (e^\rho-1)^2 \omega^n. 
\end{equation}
Here the volume $V= \int_X \omega^n$ is independent of $\omega$. 
The straightforward variational computation shows that the metric $\omega$ is a critical point iff $e^\rho-1$ is the Hamilton function for some one-parameter subgroup $\eta\colon \G_m \to \Aut(X, -K_X)$. 
It is also clear from the definition that the condition gives the Ricci-analogue of the extremal K\"ahler metric defined in terms of the classical Calabi functional. 
These two metrics are not the same while the above $\eta$ is generated by the extremal vector field. 
There as well exists the infinite-dimensional GIT picture \cite{Don15} so that the Ricci-Calabi functional can be seen as the square norm of a certain moment map.  
The role of the Kemp-Ness functional in GIT is then played by the famous D-energy 
\begin{equation*}
D(\phi) = 
-\log \frac{1}{V}\int_X e^{-\phi}
-\frac{1}{(n+1)V}\sum_{i=0}^{n} \int_X (\phi-\phi_0) \omega^i\wedge \omega_0^{n-i}. 
\end{equation*}
The definition of the D-energy first appeared in \cite{BM86} and was written down to this form by \cite{Din88}. 
The gradient flow 
\begin{equation*}
\frac{\partial}{\partial t} \phi = 1-e^\rho  
\end{equation*}
was initially studied in our previous work \cite{CHT17}. 
In \cite{His19}, \cite{Xia19} it was shown that the flow indeed minimizes $R(\phi)$ and is naturally related with the optimal degeneration of the Fano manifold. 
From now on we call the pair of the critical point of $R(\phi)$ and the one-parameter subgroup {\em Mabuchi soliton}, since it is characterized as the self-similar solution of the flow.  

Our main result claims that the existence of Mabuchi soliton is equivalent to certain algebraic stability condition. 
It extends the result of \cite{Yao17}, \cite{Nak17} for the toric case to general Fano manifolds. 
Our approach precisely follows \cite{BBJ15} where they give a new variational proof of \cite{CDS15} for a Fano manifold with finite automorphism group. 
Our first version of the preprint and \cite{Li19} extended the result to general automorphism (with $\eta=0$) case. 

\begin{thmA}
A Fano manifold $X$ admits a Mabuchi soliton if and only if it is uniformly relatively D-stable, with respect to the equivariant test configurations. 
\end{thmA} 

If the extremal vector field is zero \ie $\eta=0$, we obtain the existence result of K\"ahler-Einstein metric, with no restriction for the automorphism group. 
To obtain the result we develop the equivariant formulation which was suggested in \cite{DS16}, \cite{His18}. 
One direction deriving stability is based on Theorem \ref{slope of J_T} which was proved in \cite{His18}.  
In the equivariant setting to obtain the metric one needs special discussion particularly in Lemma \ref{rationality} and Lemma \ref{boundedness}, to control the test configurations twisted by one-parameter subgroups. 
Also we need to be careful for the equivariant formalism to confirm that the minimizer of the energy is a weak solution. This is concentrated in Theorem \ref{minimizer is a weak solution}. 
The first version of our preprint however had a serious error in the final step of the proof and this was pointed out and solved by \cite{Li19}, particularly when $\eta=0$. 
In the present version we follow \cite{Li19} in this respect. 
This corresponds to the discussion around Lemma \ref{boundedness} and \ref{non-triviality}. 
We here also serve another simple proof of Lemma \ref{boundedness}. 
We hope that we could still contribute to the formalism of the problem and further extend our scope to the Mabuchi soliton. 

The required stability condition is introduced in Definition \ref{relative uniform D-stability}.
The concept of D-stability originates from \cite{Berm16}. 
As the K-stability introduced by \cite{Don02} naturally arises from the Calabi functional and the K-energy, D-stability arises from the above Ricci-Calabi functional and the D-energy. 
Uniformity of the stability was introduced in our previous work \cite{BHJ15} and \cite{Der16a} independently. 
In regard to the torus containing the soliton vector field we may also formulate the relative version of the D-stability, especially based on the author's previous work \cite{His16a} \cite{His16b}, and \cite{His18}. 
Putting these together we formulate the {\em uniform relative stability} which reflects the coercivity of the modified D-energy. 
In fact it was shown by \cite{LZ17} (in a different formulation) that the relevant coercivity is equivalent to the existence of Mabuchi soliton. 
If we derive the coercivity from the stability, the uniformity is critical in controlling the sequence of test configurations. 
The relative consideration of the energies relies on \cite{BWN14}. 
Although they were mainlsymmetricy concerned with the K\"ahler-Ricci soliton the techniques are valid for the general situations including the present case. 


One remarkable point clarified in this paper is that we may restrict ourselves to test configurations equivariant for the {\em whole} automorphism group, in showing the existence of the metric. 
This fact owes to Theorem \ref{minimizer is a weak solution}, where we used the fact that the Futaki invariant is trivial outside of the center. 

Unlike K-stability, D-stability works only for Fano manifolds, however, as Theorem A and its proof show, the treatment is much easier. 
Existence of extremal K\"ahler metric is still open problem, even for the anti-canonical polarizations. 
A simple argument shows that the Mabuchi soliton assures the extremal K\"ahler metric. 
A new circumstance in the relative setting is that the two metrics are in fact not equivalent. 
The first counterexample is raised in the latest version of \cite{NSY17}. 

Compared with K\"ahler-Ricci soliton, Mabuchi soliton has in some sense more algebraic nature. For example the soliton vector field is periodic and actually generates $\eta$.  
On the other hand, the gradient flow is not so flexible as the K\"ahler-Ricci flow. 
In addition toric examples in \cite{NSY17} is in contrast to the result of \cite{WZ04}. 

Along the variational approach we may naturally understand the uniqueness of Mabuchi soliton. 

\begin{thmB}[\cite{Mab03}, Theorem C]
Let $(\omega_0, \eta_0)$ and $(\omega_1, \eta_1)$ be smooth Mabuchi solitons. 
Then there exists an automorphism $f \in \Aut^0(X)$ in the identity component such that 
$f^* \omega_1 =\omega_0$, $f^*\eta_0 = \eta_1$. 
\end{thmB} 

Our argument also gives a new proof of the Matsushima-type theorem in \cite{Mab03}, \cite{Nak18}. 
Namely, if a Fano manifold admits the Mabuchi soliton, the identity component of the group of automorphism preserving the extremal vector field is reductive.  
These uniqueness and the reductivity are key materials for the derivation of the coercivity from existence of the metric. 


\vspace{5mm}
\begin{ackn}
The author express his gratitude to Professor C. Li, for his kind communication especially pointing out the serious error in our first version of the preprint.  
The author wishes to thank Professor R.~Berman, S.~Boucksom, and M.~Jonsson for very fruitful discussions. 
Especially for the equivariant formulation he learned a lot from the three professors, as in our previous paper \cite{His18}. 
We are also grateful to E. Inoue, S. Saito and R. Takahashi for helpful comments. 
This research was supported by JSPS KAKENHI Grant Number 15H06262 and 17K14185. 
\end{ackn}


\section{Mabuchi soliton and modified D-energy} 

\subsection{Notation}\label{notation}

Throughout the paper $X$ denotes an $n$-dimensional Fano manifold and a K\"ahler metric $\omega$ is taken in the first Chern class $c_1(X)$. 
We adopt the additive notation writing the anti-canonical bundle as $-K_X$ and the fiber metric as $\phi$. 
While we do not fix a specific covering $\{ U_\a\}_\a$ of local coordinate patches $U_\a$ with index $\a$, the symbol $\phi$ is interpreted to a function $\phi_\a$ on each $U_\a \subset X$.  
In a local frame any section $s$ of $-K_X$ is identified with a function $s_\a$ and it is evaluated by the multiplication of $e^{-\phi_\a}$ to $\abs{s_\a}^2$.  
On the intersection $U_\a \cap U_\b $ for two indicies $\a, \b$ and coordinates $z_\a^i, z_\b^j$ $1 \leq i, j \leq n$ the transition function is written to $g_{\a \b} = \det{\big[\frac{\partial z_\a^i}{\partial z_\b^j}\big]_{ij}}$ and the compatibility $\phi_\a = \phi_\b +\log \abs{g_{\a \b}}^2$ holds. 
If we put $d^c = \frac{\partial-\bar{\partial}}{4\pi\sqrt{-1}}$, it follows that the Chern curvature $\omega_\phi=dd^c\phi$ is globally well-defined. 
We set $\cH(X, -K_X)$ as the collection of smooth fiber metric $\phi$ on $-K_X$ such that $\omega_\phi$ is positive. 
By $dd^c$-Lemma, any metric $\omega$ in $c_1(X)$ equals to $\omega_\phi$ for some $\phi \in \cH(X, -K_X)$ which is unique up to addition of a constant. 
For this reason $\cH(X, -K_X)$ is called {\em the space of K\"ahler metrics}. 
We essentially need $\phi$ instead of $\omega=\omega_\phi$ in order to consider the action of the Hamilton diffeomorphism group.   

\subsection{Ricci curvature formulation}\label{Ricci curvature formulation} 
 
We briefly review some energy formulations to the K\"ahler-Einstein problem, which make use of the Ricci potential. 
There has been another (and probably major) scalar curvature formulation which works for a general polarized manifold. 
In terms of the scalar curvature one may introduce the Calabi functional and notion of K-stability observing the behavior of the K-energy along the degeneration of the manifold. 
See the milestone works \cite{Cal82}, \cite{Cal85}, \cite{Fut83}, \cite{Mab86}, \cite{Tian97}, and \cite{Don02}. 
D-energy which we will explain is as well classical but the determination of the corresponding D-stability \cite{Berm16} and the momentum map picture \cite{Don15} were rather recent. 

Let us start from defining two probability measures associated to a K\"ahler metric $\omega=\omega_\phi$, or equivalently $\phi \in \cH(X, -K_X)$. 
One is the Monge-Amp\`ere measure $V^{-1}\omega_\phi^n$. 
The other one which we call the {\em canonical measure} is special for the Fano case and defined to be 
\begin{equation}\label{canonical measure}
\mu_\phi := \frac{e^{-\phi}}{\int_X e^{-\phi}}, 
\end{equation}
where $e^{-\phi}$ denotes the global volume form described as $e^{-\phi_\a} \Wedge_{i=1}^n  dz_\a^i \wedge d\bar{z}_\a^i$ on a coordinate patch $U_\a$. 
Note that the metric is K\"ahler-Einstein iff it satisfies the Monge-Amp\`ere equation $V^{-1}\omega_\phi^n=\mu_\phi$. 
Therefore we are lead to focus on the difference of these two measures. 
In fact it precisely gives the infinite-dimensional moment map. 
Namely, once we regard a fixed K\"ahler metric $\omega$ as a symplectic form  and instead collect all the complex structures $J$ compatible with $\omega$, one may attach to each $J$  the measure 
\begin{equation}\label{infinite dimensional moment map}
J \mapsto 
\mu_\phi - V^{-1}\omega^n. 
\end{equation}
The group of Hamilton diffeomorphisms naturally acts on the complex structures.  
The Lie algebra of this group is naturally identified with smooth function space $C^\infty(X; \R)$ with Poisson bracket and hence the above defines a map to the dual Lie algebra. It indeed satisfies the moment map condition. 
More precisely, we should impose to $J$ the compatibility condition with the fiber metric $\phi$, but see \cite{Don15} for the detail explanation. 

The square norm of the moment map is written down to 
\begin{equation}\label{Ricci-Calabi functional}
R(\phi) = \frac{1}{V}\int_X (e^\rho-1)^2 \omega^n,  
\end{equation}
which we call the Ricci-Calabi functional. 
Our interest is the critical point of the Ricci-Calabi functional which gives a generalization of the K\"ahler-Einstein metric. 
The first variation of $R\colon \cH(X, -K_X) \to \R$ is given as follows. 
See also \cite{Nak18} for calculating the second variation. 

\begin{prop}[\cite{CHT17}, Proposition 2.3]
Set the twisted Laplacian on functions $f \in C^{\infty}(X,\mathbb{C})$  by
\begin{equation}\label{twisted Laplacian}
L_{\rho}f = \Delta_\omega f + (\bar{\partial} \rho, \bar{\partial} f)_{\omega}. 
\end{equation}
Then the first variation of the Ricci-Calabi functional is given as  
\begin{equation}\label{variation of Ricci-Calabi functional}
\delta R(\phi) = -\frac{2}{V} \int_{X} \delta \phi \left( L_{\rho}\tilde{f} + \tilde{f} \right) d\mu_\phi, 
\end{equation}
where 
\begin{equation*}
\tilde{f} = (e^{\rho}-1)  - \frac{1}{V}\int_{X}(e^{\rho}-1) d\mu_\phi. 
\end{equation*}
\end{prop} 

As a consequence, $\phi$ is the critical point of the Ricci-Calabi functional iff $e^\rho -1$ is a Hamilton function. 
One can check this by a simple application of the Bochner-Kodaira formula. 
Since $X$ is Fano any holomorphic vector field defines a function $h$ unique up to addition of a constant such that 
\begin{equation}\label{Hamilton function}
\sqrt{-1}\bar{\partial} h = i_v \omega. 
\end{equation}
We call $h$ a Hamilton function. 
\begin{dfn}
A K\"ahler metric $\omega \in c_1(X)$ is called a Mabuchi soliton if $e^\rho -1$ is a Hamilton function for some holomorphic vector field. 
\end{dfn} 
The vector field is zero iff $\rho=0$ and in this case Mabuchi soliton is nothing but K\"ahler-Einstein. 

Going back to the moment map picture, we also have the canonical energy functional $D\colon \cH(X, -K_X) \to \R$ with the outer derivative $(dD)_\phi=\mu_\phi- V^{-1}\omega^n$ at $\phi$. 
We call it D-energy. It is in fact separated into two terms $D=L-E$ and each term is specifically defined as  
\begin{equation}\label{definition of L and E}
L(\phi) := -\log \frac{1}{V} \int_X e^{-\phi}, ~~~
E(\phi) := \frac{1}{(n+1)V} \sum_{i=0}^n \int_X (\phi-\phi_0) \omega^i \wedge \omega_0^{n-i}. 
\end{equation}
We here take a reference $\phi_0$ and $\omega_0=dd^c\phi_0$. 
Note that the difference $\phi-\phi_0$ defines a global function while $\phi$ does not. 
One can easily compute to check the differential 
\begin{equation}\label{derivative of L and E}
(dL)_\phi = \mu_\phi, ~~~
 (dE)_\phi =V^{-1} \omega^n. 
\end{equation}
The definition of the Monge-Amp\`ere energy $E$ chose $\phi_0$ but it is characterized by (\ref{derivative of L and E}), up to addition of a constant.  

\subsection{Space of finite energy metrics}\label{space of finite energy metrics}

A fundamental property of the D-energy is that it is convex along any geodesic in the space of K\"ahler metrics. 
Since the difference $\phi-\psi$ of any two $\phi, \psi \in \cH(X, -K_X)$ defines a global function, tangent space at any point of $\cH(X, -K_X)$ is identified with $C^\infty(X; \R)$. 
Mabuchi's inner product \cite{Mab87} for any tangents $u, v \in C^\infty(X; \R)$ at $\phi$ is 
\begin{equation}\label{Mabuchi inner product}
\langle u, v\rangle =\frac{1}{V}\int_X uv \omega^n.  
\end{equation}
Any curve $\phi^t$ $(t \in [a, b])$ in $\cH(X, -K_X)$ defines a function 
$\Phi(\tau, x):=\phi^{-\log\abs{\tau}}(x)$ 
of complex variables $e^{-b} \leq \abs{\tau} \leq e^{-a}$ and $x\in X$. 
It is well-known (from \cite{Sem92}) that the geodesity for (\ref{Mabuchi inner product}) is equivalent to the degenerate Monge-Amp\`ere equation 
\begin{equation}\label{degenerate MA equation} 
(dd^c_{\tau, x}\Phi)^{n+1} =0. 
\end{equation}
The left-hand side at the same time describes the Monge-Amp\`ere energy by the fiber integration formula 
\begin{equation}\label{Hessian of E}
dd^c_\tau E(\phi) = \int_X (dd^c_{\tau, x}\Phi)^{n+1}. 
\end{equation} 
It follows that $E$ is affine along any geodesics. 
In fact for given smooth endpoints the bounded {\em weak} geodesic $\Phi$ connecting them uniquely exists, but it is not $C^2$ in general. 

Variational approach even requires the appropriate completion of the space of smooth metrics. 
These facts strongly motivate to consider a singular fiber metric $\phi$ which is only locally integrable and satisfies $dd^c \phi \geq  0$ in the sense of current. 
We denote the collection of all such singular $\phi$ by $\PSH(X, -K_X)$. 
It equivalent to say that in a coordinate patch $U_\a$, $\phi_\a$ is  pluri-subharmonic (psh for short) function. 
For the bounded psh function the wedge product of the current $\omega_\phi^n=(dd^c\phi)^n$ is safely defined thanks to the celebrated work of \cite{BT76}. 
In particular we may define the Monge-Amp\`ere energy $E$ for locally bounded $\phi$.  
For a smooth boundary data we have the bounded solution of (\ref{degenerate MA equation}). 
From the recent result \cite{CTW17}, the solution is actually of $C^{1, 1}$.   

The Monge-Amp\`ere operator $\phi \mapsto V^{-1}\omega_\phi^n$ can not be continuously extended to $\PSH(X, -K_X)$.
Following \cite{BEGZ10} and \cite{BBGZ13}, one can however take the reference $\phi_0$ smooth and bounded approximation $\phi^{(j)} := \max\{\phi, \phi_0-j\}$ of $\phi \in \PSH(X, -K_X)$, to define the {\em non-pluripolar Monge-Amp\`ere measure } 
\begin{equation}\label{non-pluripolar MA}
\MA(\phi) 
:= \lim_{ j \to \infty} 1_{\{ \phi > \phi_0-j \}}V^{-1}\omega_{\phi^{(j)}}. 
\end{equation} 
By the construction $\MA(\phi)$ drops the mass of the unbounded locus so it is no longer a probability measure.    
It can be further shown that $\MA(\phi)$ is local in pluri-fine topology and has no mass on any pluripolar set. 
In a similar idea taking bounded $\psi$ such that $\psi \geq \phi $ locally we define the Monge-Amp\`ere energy as 
\begin{equation}\label{E for arbitrary psh}
E(\phi) := \inf_\psi E(\psi) \in \R \cup \{ -\infty\}. 
\end{equation}
The extended Monge-Amp\`ere energy is upper-semicontinuous in the $L^1$-topology of $\PSH(X, -K_X)$. 
Moreover, the level set $\{ E \geq C\} $ is compact in this weak topology, as a non-linear analogue of the Banach-Alaoglu theorem. 

Let us now consider $p \geq 1$ and the $L^p$-Finsler distance $d_p$ of  $\cH(X, -K_X)$, defined by the norm of tangents 
\begin{equation}\label{norm}
\norm{u}_p :=\bigg[ \frac{1}{V}\int_X \abs{u}^p \omega_\phi^n  \bigg]^{\frac{1}{p}}. 
\end{equation} 
As we shall see, $p=1$ plays the special role in the variational approach.  

\begin{thm}[\cite{Dar15}, Theorem 2. See also \cite{BBJ15}, Theorem 1.7]\label{completeness}
Take a smooth non-increasing sequence of $\phi_j \in \PSH(X, -K_X)$ converges to $\phi$. 
Endow the space of metrics with finite Monge-Amp\`ere energy 
\begin{equation}\label{space of finite MA energy metrics}
\cE^1(X, -K_X) := \bigg\{ \phi \in \PSH(X, -K_X): E(\phi) > -\infty \bigg\} 
\end{equation} 
with the distance
\begin{equation}\label{d_1}
d_1(\phi, \psi):= 
\lim_{j \to \infty} d_1(\phi_j, \psi_j). 
\end{equation}
It then gives the coarsest refinement of the $L^1$-topology so that $E$ is continuous. 
The Monge-Amp\`ere energy is affine along every geodesic on $(\cE^1(X, -K_X), d_1)$. 
Moreover, $(\cE^1(X, -K_X), d_1)$ realizes the completion of $(\cH(X, -K_X), d_1)$. 
\end{thm}

We usually refer to the $L^1$-topology as the ``weak" topology and the $d_1$-topology as the ``strong" topology. 
The space $\cE^1(X, -K_X)$ is contained in the {\em finite energy class} 
\begin{equation}\label{space of finite energy metrics}
 \cE(X, -K_X) := \bigg\{ \phi \in \PSH(X, -K_X): \int_X \MA(\phi)=1 \bigg\}. 
\end{equation}
Restricted to $\cE(X, -K_X)$ the non-pluripolar Monge-Amp\`ere operator is continuous along any monotone sequence. 
The determination of the domain of Monge-Amp\`ere operator owes to the pioneering work \cite{Ceg98}. The compact setting is treated in \cite{GZ07},  \cite{BEGZ10}. See also the comprehensive textbook \cite{GZ17}. 

As it was shown in \cite{Dar17a}, 
geodesics connecting two points are not unique in $\cE^1$, however, for any $\phi^0, \phi^1 \in \cE^1(X, -K_X)$ there exists the canonical one if we impose the plurisubharmonicity to the corresponding function $\Phi$ on the product space. 
\begin{thm}[\cite{BBJ15}, Theorem 1.7] 
For any $\phi^0, \phi^1 \in \cE^1(X, -K_X)$ the solution of (\ref{degenerate MA equation}) has the unique solution $\phi^t \in \cE^1(X, -K_X)$ and provides a geodesic for $d_1$. 
This special type of geodesic is called {\em psh geodesic}. 
\end{thm} 
The convexity of the D-energy functional along such a weak geodesic is established  by the fundamental work \cite{Bern09}, \cite{BP08}, and \cite{Bern11}. 
Since it is not scale free, \ie $E(\phi+c)=E(\phi)+c$ for constants $c$, it is convenient to introduce the Aubin's J-functional: 
\begin{equation}\label{definition of J}
J(\phi) =L_0(\phi)-E(\phi):=  \sup_X(\phi-\phi_0)-E(\phi). 
\end{equation} 
It follows that $J(\phi)- d_1(\phi, \phi_0)$ is uniformly bounded. 
As we have the uniform estimate 
\begin{equation}
\sup_X(\phi-\phi_0) \leq \frac{1}{V} \int_X (\phi-\phi_0) \omega_0^n  +C 
\end{equation} 
for $\phi \in \cH(X, -K_X)$, sometimes $ V^{ -1} \int_X (\phi-\phi_0) \omega_0^n$ is adopted for the definition of $J$. 
Moreover one may introduce the symmetric I-functional 
\begin{equation}\label{I-functional}
I(\phi, \psi)
:= \int_X (\phi-\psi)(\MA(\psi)-\MA(\phi))
\end{equation}
which satisfies the pseudo-triangle inequality: 
\begin{equation}\label{pseudo-triangle inequality}
c_n I(\phi_1, \phi_2) \leq I(\phi_1, \phi_2)+I(\phi_2, \phi_3) 
\end{equation}
for any $\phi_1, \phi_2, \phi_3 \in\cE^1(X, \omega_0)$. 
The induced topology of $\cE^1$ is equivalent to $d_1$-topology. 
See \cite{BBEGZ16}, section 1.  

We say that the D-energy is coercive if 
\begin{equation*}
D(\phi) \geq \e J(\phi) -C 
\end{equation*}
for any smooth $\phi$. 
From the weak compactness of the level set $\{ E \geq -C\}$ the coercivity guarantees the existence of a minimizer. 

\begin{rem} 
If one considers coercivity for $d_2$ there is no example of Fano manifolds satisfy the condition. 
This is confirmed by \cite{BHJ15}, Proposition 8.5 for the K-energy. 
We may check the same for the D-energy, using Definition \ref{definition of D^NA}. 
\end{rem}

\subsection{Modified D-energy}\label{modified D-energy}

Exploiting the inner product (\ref{Mabuchi inner product}) we may also modify the D-energy such that the critical point gives the Mabuchi soliton. 

It is consistent to consider the group of bundle automorphism $\Aut(X, -K_X)$, indeed any $g \in \Aut(X, -K_X)$ pulls-back $\phi \in \PSH(X, -K_X)$ to $g^*\phi$. 
More precisely, for any $x \in X$ a vector $v \in (-K_X)_x $ is evaluated as 
\begin{equation}\label{pull-back}
\abs{v}^2 e^{-(g^*\phi)(x)} = \abs{g\cdot v}^2 e^{-\phi(gx)}. 
\end{equation}
Note that the local frame identifying the function $\phi(gx)$ with the pull-backed fiber metric depends on $g$. 
Indeed we see from (\ref{pull-back}) that the function $\phi(gx)$ is unbounded in $g$. 
Since the line bundle is anti-canonical, any automorphism of X can be lifted to $-K_X$, hence the group splits into 
$\Aut(X, -K_X)= \Aut(X) \times \G_m$. 
We denote the identity component by $\Aut^0(X, -K_X)$.  
In particular, constant multiplication on each fiber defines the identical one-parameter subgroup which we denote by $1\colon \G_m\to \Aut^0(X, -K_X)$. 

Our first step is to specify the soliton vector field of the Mabuchi soliton. 
For the purpose we fix an algebraic subtorus $T \subset \Aut(X, -K_X)$. 
The compact part $S=\Hom(\S^1, T)$ is canonically defined. 
Henceforth we take a compact subgroup $K \subset \Aut^0(X, -K_X)$ which contains $S$ and commutes with $T$. It is possible that $K=S$. 
We define the space of $K$-invariant metrics 
\begin{equation}
\cH(X, -K_X)^K:= 
\bigg\{ \phi \in \cH(X, -K_X): g^*\phi= \phi ~~\text{for any $g \in K$ }\bigg\}. 
\end{equation} 

From the assumption that $K$ commutes with the torus, $T$ acts on $\cH(X, -K_X)^K$. 
Tangents are identified with smooth $K$-invariant functions. 
We take $K$-invariant functions in \ref{norm} to define the distance $d_1$. 
Similarly the space of finite energy $K$-invariant metrics $\cE^1(X, -K_X)^K$ can be defined  and has the same property as Theorem \ref{completeness}. 
Note that in the definition 
\begin{equation*}
d_1(\phi_0, \phi_1)
= \inf_{\phi_t} \int_0^1 \frac{1}{V}\int_X \abs{\dot{\phi}_s} \omega_{\phi_t}^n  ds
\end{equation*}
a path $\phi_t$, connecting $\phi_0$ to $\phi_1$, is taken as $K$-invariant. 
The unique psh geodesic is however $K$-invariant if $\phi_0, \phi_1 \in \cH(X, -K_X)^K$. 
It follows that $d_1$ for $\cH(X, -K_X)^K$ equals to the previous one for $\cH(X, -K_X)$. 

Let us denote by $N := \Hom(\G_m, T)$ the lattice of all one-parameter subgroups $\mu \colon \G_m \to T$. 
The dual lattice $M := \Hom(T, \G_m)$ is identified with the set of characters. 
Observe that the vector space $N_\R := N\otimes \R$ is identified with the Lie algebera $\fs$ of $S$. 
From the basic symplectic geometry $S$ defines the moment polytope $P \subset M_\R$ as the image of the moment map 
\begin{equation}
m_\phi\colon X \to M_\R. 
\end{equation}
Actually for any smooth $K$-invariant $\omega_\phi$ and $\mu \in N_\R$ we have the unique map satisfying 
\begin{equation}
\langle \mu, m_\phi(x) \rangle =\frac{d}{dt} \bigg\vert_{t=0} \phi(\mu(e^{t}) x). 
\end{equation}
It is easy to show that $m_\phi$ is independent of the metric. 
Once $\mu \in N_\R$ is fixed $h_\mu:= \langle \mu, m_\phi(x) \rangle$ gives the (unnormalized) Hamilton function. 
Notice that when $\mu \in N$ is generated by a vector field $v \in \ft$ we have the relation (\ref{Hamilton function}). 
The $S$-invariance of $\omega$ guarantees that $h_\mu$ is real. 
For the identical one-parameter subgroup we observe $h_1=1$. 

In this convention following \cite{FM95} we introduce the inner product 
\begin{equation}\label{inner product of Hamilton functions}
\langle \mu, \nu \rangle:= 
\int_X h_\mu h_\nu \omega^n
\end{equation}
for $\mu, \nu \in N_\R$. 
Of course $h_\mu$ depends on the choice of metric $\omega$ but as we will see in the next section the above inner product is determined only by $\mu, \nu$. 
The Hamilton function can be regarded as the tangent vector of the associated (smooth) geodesic ray 
\begin{equation}
\phi^t = \mu(e^{-t})^*\phi^0 
\end{equation} 
for a given initial $\phi^0 \in \cH(X, -K_X)^K$. 
Note again that $\phi^t$ is $K$-invariant for each $t$ since we assumed $K$ commutes with $T$. 

The slope of D-energy along this ray is independent of $t \in [0, \infty)$ and explicitly computed as 
\begin{equation}\label{Futaki invariant}
F(\mu):= \frac{1}{V}\int_X h_\mu (e^\rho -1) \omega^n. 
\end{equation}
It is precisely the classical Futaki invariant \cite{Fut83} for the vector field generating $\mu$. 
Notice that using the scalar curvature $S_{\omega_\phi}$ and its average $\hat{S}$ the Futaki invariant can be also written as 
\begin{equation}\label{L-Futaki invariant}
F(\mu)= \frac{1}{V}\int_X h_\mu(S_{\omega_\phi}-\hat{S}) \omega^n 
\end{equation}
and the description leads us to the definition of the K-energy. 
Therefore, D-energy is rather natural in view of the original description in \cite{Fut83}. 
The extremal vector field naturally arises from the optimization of $F(\mu)$ normalized by $\norm{\mu}= \langle \mu, \mu \rangle^{\frac{1}{2}}$. 
Actually a simple variational computation 
\begin{equation*}
\d \bigg( \frac{F(\mu)}{\norm{\mu}} \bigg)
= \frac{1}{\norm{\mu}} \bigg( F (\d \mu) -\frac{\langle \d\mu, \mu \rangle}{\langle \mu, \mu \rangle}F(\mu)\bigg) 
\end{equation*}
suggest us to introduce the extremal one-parameter subgroup $\eta \in N_\R$, which satisfies 
\begin{equation}\label{extremal vector} 
F(\mu)- \langle \mu, \eta \rangle =0 
\end{equation}
for any $\mu \in N_\R$. 
Since (\ref{extremal vector}) is a system of linear equations, one easily see that $\eta$ is uniquely characterized by the above relation. 
It is also easy to check $\eta \in N_\Q$ and automatically 
\begin{equation}\label{extremal is orthogonal to constant}
\langle 1, \eta \rangle= \int_X h_\eta \omega^n =0. 
\end{equation} 
On the other hand the Mabuchi soliton should minimize $R(\phi)$.  
In fact if there exists a Mabuchi soliton $\omega_\phi$ with $e^\rho -1 =h_\mu $ for some $\mu \in N_\R$ we have $\mu =\eta$ and 
\begin{equation*}
R(\phi) = \frac{F(\eta)}{\norm{\eta}}. 
\end{equation*}
That is, the both optimizer $\phi$ and $\eta$ attain the same value. 
In general the lower bound of the Ricci-Calabi functional is attained by the normalized non-Archimedean D-energies which we introduce in the next section.  
See the recent work \cite{Xia19} and \cite{His19} for this topic. 
The following is a consequence of Theorem \ref{property of E_g} in the next subsection. 

\begin{prop}
There exists the modified Monge-Amp\`ere energy $E_\eta \colon \cH(X, -K_X)^K \to \R $ satisfying 
\begin{equation*}
(d E_\eta )_\phi = 
 (1+h_\eta) V^{-1}\omega_\phi^n 
\end{equation*}
at each point $\phi \in \cH(X, -K_X)^K$. Moreover, $E_\eta$ is geodesically affine.  
\end{prop} 

We define modified D-energy as $D_\eta:= L-E_\eta$. 
It follows from the proposition that a smooth metric $\omega$ is Mabuchi soliton iff it is a critical point of the modified D-energy. 

\subsection{Modified Monge-Amp\`ere measure}\label{modified Monge-Ampere measure}

For the variational approach it is necessary to handle with $E_\eta(\phi)$ for singular $\phi$. 
In this part following \cite{BWN14} we discuss basic properties of the modified Monge-Amp\`ere measure. 
In \cite{BWN14} the case $K=S$ is considered but the same argument works for general $K$ which contains $S$ and commuts with $T$. 
Let continuously $m_\phi \colon X \to P$ be the moment map. 

\begin{dfn}[\cite{BWN14}]\label{modified MA measure}
Let $\phi \in \cH(X, -K_X)^K$.  
For a non-negative continuous function $g\colon P \to \R$ define the modified Monge-Amp\`ere measure 
\begin{equation*}
\MA_g(\phi):= g(m_\phi(x))\MA(\phi). 
\end{equation*}
The definition further extends to general $\phi \in \PSH(X, -K_X)^K$ so that 
the measure $\MA_g(\phi)$ is 
local in plurifine topology and non-pluripolar. 
\end{dfn} 

\begin{thm}[\cite{BWN14}, Theorem 2.7]\label{DH measure}
The Duistermaat-Heckman measure 
\begin{equation*}
\DH_T:= (m_\phi)_* \MA(\phi)
\end{equation*} 
is independent of smooth $\phi$ and defines a positive measure on $M_\R$. 
For any $\phi \in \PSH(X, -K_X)^K$ we have 
\begin{equation*}
\int_X \MA_g(\phi) \leq \int_P g \DH_T. 
\end{equation*} 
The equality holds if $\phi \in \cE(X, -K_X)^K$, namely when $\MA(\phi)$ is a probability measure. 
\end{thm} 

Let $\langle, \rangle$ be the canonical paring of the lattices $N$ and $M$. We are interested in the case 
\begin{equation}\label{choice of g}
g(x):= 
1+\langle \eta, x \rangle -\int_P \langle \eta, x \rangle \DH_T. 
\end{equation} 
Note that $g$ of this form is not necessarily non-negative. 
At least when $g \geq 0$ and $\phi$ smooth we observe 
$\MA_g(\phi)= (1+h_\eta)\MA(\phi)$ and 
\begin{equation}\label{bound of g-MA} 
(\inf_P g) \MA(\phi) \leq \MA_g(\phi) \leq (\sup_P g) \MA(\phi). 
\end{equation} 
Notice that $1+h_\eta >0$ holds if $X$ admits a Mabuchi soliton. 
As in the next section we shall see that the condition $1+h_\eta >0$ is numerical, from now on we assume that the above $g$ is positive. 
Then the equation of Mabuchi soliton may be interpreted into the Monge-Amp\`ere type equation 
\begin{equation}\label{weak Mabuchi soliton}
\MA_g(\phi) =\mu_\phi. 
\end{equation}
We call $\phi \in \PSH(X, -K_X)^K$ satisfying this condition a {\em weak Mabuchi soliton}.  

On the other hand, if we choose  
\begin{equation}\label{choice of g for KRS}
g(x)= 
\frac{e^{\langle \mu, x \rangle}}{\int_P e^{\langle \mu, x \rangle}  \DH_T}  
\end{equation} 
with certain $\mu$, equation (\ref{weak Mabuchi soliton}) gives the weak K\"ahler-Ricci soliton. In this case $g$ is always positive but $\mu \notin N_\Q$. 

\begin{thm}[\cite{BWN14}, Lemma 2.14, Proposition 2.15]\label{property of E_g}
We have the canonical energy $E_g\colon \cH(X, -K_X)^K \to \R$ such that 
$(dE_g)_\phi =\MA_g(\phi)$. 
For general $\phi \in \PSH(X, -K_X)^K$ we have 
\begin{equation*}
E_g(\phi) := \inf_{\psi \geq \phi} E_g(\psi ), 
\end{equation*}
where $\psi$ runs through bounded ones, or $\cH(X, -K_X)^K$. 
The functional $E_g$ is monotone, upper-semicontinuous in $L^1$-topology,  and continuous for any non-increasing sequence in $\PSH(X, -K_X)^K$.  
\end{thm} 

We would write $E_\eta := E_g$ in the case (\ref{choice of g}). 
The description of $E_g$ is easily specified so we briefly sketch it. 
For the path $\phi_t= (1-t)\phi +t \phi_0$ the demanded $E_g$ is computed as 
\begin{align*}
E_g(\phi) 
&= \int_0^1\frac{d}{dt} E_g(\phi_t) dt 
=\frac{1}{V}\int_0^1 dt \int_X (\phi-\phi_0)g(m_{\phi_t})\omega_{\phi_t}^n \\
&=\frac{1}{V}\sum_{i=0}^n \bigg(\begin{matrix} n \\ i \end{matrix}\bigg) 
\int_0^1t^i(1-t)^{n-i} dt \int_X (\phi-\phi_0) g(m_{\phi_t}) \omega_\phi^i \wedge \omega_0^{n-i}. 
\end{align*}
Note $m_{\phi_t}=t m_\phi +(1-t)m_{\phi_0}$ and that the last integrant is just a variant of modified Monge-Amp\`ere measure. 
Therefore we may exploit Definition \ref{modified MA measure} to derive the required property of $E_g$. 
If $\inf_P g$ is positive it follows  
\begin{equation}\label{E_g vs E}
(\sup_P g) E(\phi) \leq E_g(\phi) \leq (\inf_P g)E(\phi) 
\end{equation} 
provided $\sup_X (\phi-\phi_0)=0$.  
It implies that $E_g(\phi) > -\infty$ if $\phi$ has finite Monge-Amp\`ere energy. 
At any case we define the $g$-modified J-energy by 
\begin{equation}
J_g(\phi):= L_0(\phi)-E_g(\phi).
\end{equation}
Let again $J_\eta := J_g$ in the case (\ref{choice of g}). 
This is after all equivalent to the ordinal $J$-functional. 

\begin{lem}\label{J_g vs J}
When $g>0$, we have 
\begin{equation*}
(\inf_P g) J_g(\phi) \leq J(\phi) \leq (\sup_P g) J_g(\phi) 
\end{equation*}
for all $\phi \in \cH(X, -K_X)^K$. 
\end{lem} 
\begin{proof}
Since $J_g(\phi+c)=J_g(\phi)$ for any constant $c \in \R$, we may assume $\sup_X(\phi-\phi_0)=0$. 
The claim is then a consequence of (\ref{E_g vs E}). 
\end{proof} 
 
For a given probability measure $\mu$ one can consider the Monge-Amp\`ere type equation $\MA_g(\phi)=\mu$. 
It was also shown in \cite{BWN14} Theorem $2.18$ that 
there exists the unique solution $\phi \in \cE^1(X, -K_X)^K$ iff 
the Legendre dual of the Monge-Amp\`ere energy 

\begin{equation}\label{E_g^*}
E_g^*(\mu) := 
\sup_{\phi \in \cE^1(X, -K_X)} \bigg[ E_g(\phi) -\int_X (\phi -\phi_0) d\mu\bigg]
\in \R \cup \{ \infty \} 
\end{equation}

is finite. Moreover, the above supremum is attained by the solution. 
We denote the dual of $E_\eta$ by $E_\eta^*$. 

What we will study is the $g$-modified D-energy $D_g(\phi) := L(\phi)-E_g(\phi)$ and the equation (\ref{weak Mabuchi soliton}). 
The convexity of $L$-functional follows from the main result of \cite{Bern11}, while 
a direct computation shows the fiber integration formula:  
\begin{equation}\label{fiber integration formula for E}
dd^c_\tau E_g(\phi^t) = \frac{1}{(n+1)V} \int_X g(m_\Phi)(dd^c_{x, \tau} \Phi(x, \tau))^{n+1}, 
\end{equation}
which generalizes (\ref{Hessian of E}). 
We observe that $E_g$ is affine along any psh geodesic. 
Moreover, if $g>0$, the affineness of $E_g(\phi^t)$ and $E(\phi^t)$ are equivalent for any ray $\phi^t$ on $\cE(X, \omega_0)$. 
In particular the affineness of $E_g(\phi^t)$ implies that $\phi^t$ is a psh geodesic. 
As a consequence we obtain the convexity of the modified $D$-energy. 

\begin{thm}[\cite{Bern11}, \cite{BWN14} ]\label{D is convex} 
The $g$-modified D-energy $D_g(\phi) := L(\phi)-E_g(\phi)$ is convex along any weak geodesic $\phi^t$ ($t \in [a, b]$) which is a bounded solution of (\ref{degenerate MA equation}). 
\end{thm} 

We will prepare the following lemma which can be seen as a variant of Theorem \ref{D is convex}. 

\begin{lem}\label{d_g is affine}
Let $D_\eta \colon \cE^1(X, \omega_0)^K \to \R$ be the modified D-energy. 
For each $\phi \in \cE^1(X, -K_X)^K$, the map ${\mathbf d}_\eta \colon K_\C \to \C$ defined by ${\mathbf d}_\eta(g):= D_\eta(g^* \phi) $ is pluriharmonic. 
In particular if $D_\eta$ is bounded from below, then ${\mathbf d}_\eta$ is constant on the center. 
\end{lem}
\begin{proof}
This is similar to \cite{His18}, Theorem $1.6$ and Remark $2.6$ which are for the case $\eta=0$.  
The statement simply interprets geodecically affineness of $D_\eta$ into the complex variables. 

The log part is obviously pluriharmonic. 
We show that ${\mathbf e}_\eta(g):=E_\eta(g^*\phi)$ is pluriharmonic. 
Let us take an arbitrary holomorphic map $g \colon \Delta \to \Aut^0(X, L)$ which sends $z\in\Delta$ in the one dimensional disk to the automorphism $g(z)$. 
The formula \ref{fiber integration formula for E} is translated to  
\begin{equation*}
dd^c E_\eta(\phi_{g(z)}) = \frac{1}{(n+1)V} \int_X (1+h_{\phi_g, \eta})(dd^c_{z, x} \phi_{g(z)}(x))^{n+1}. 
\end{equation*}
For the holomorphic map $F\colon \Delta\times X \to X$ by $F(z, x):=g(z)\cdot x$ we have   
\begin{equation*}
(dd^c_{z, x} \phi_{g(z)}(x))^{n+1} 
= (dd^c_{z, x} F^*\phi)^{n+1} 
=F^* (dd^c_{x} \phi)^{n+1} =0.  
\end{equation*}
It implies that  ${\bf e}$ is pluriharmonic. 
\end{proof} 

Using the above lemma let us show that the weak minimizer of the energy is actually the Mabuchi soliton. This forms one of the critical step in our proof of Theorem A. 

\begin{thm}\label{minimizer is a weak solution}
Assume that $T$ contains the center of the complexified Lie group $K_\C$. 
If there exists a minimizer of modified D-energy $D_\eta \colon \cE^1(X, \omega_0)^K \to \R$ 
it defines a weak solution of (\ref{weak Mabuchi soliton}). 
\end{thm} 
\begin{proof}
For a function $v$ we define the point-wise upper envelope 
\begin{equation*}
Pv := \sup \bigg\{ \psi \in \PSH(X, -K_X)^K, \psi \leq v  \bigg\}. 
\end{equation*}
The proof is due to the highly non-trivial derivation formula (\cite{BWN14}, Proposition 2.16):  
\begin{equation*}
\frac{d}{dt}\bigg\vert_{t=0} E_g(P(\phi+tu))
=\int_X u\MA_g(\phi) 
\end{equation*}
for $\phi \in \cE^1(X, -K_X)^K$, $u \in C^0(X; \R)^K$. 
This was first established in \cite{BB10} for $g=1$, $K=\{ \id \}$ case. 
If $\phi$ is a minimizer of $D_\eta $, we observe 
\begin{align*}
f(t)&:= L(\phi+tu) -E_\eta(P(\phi+tu))\\
&\geq L(P(\phi+tu)) -E_\eta(P(\phi+tu)) \\
&\geq L(\phi) -E_\eta(\phi) =f(0). 
\end{align*} 
The derivation formula yields $f'(0)=0$ and hence 
\begin{equation}\label{for K-invariant u}
\int_X u\MA(\phi) =\int_X u\mu_\phi. 
\end{equation}
for every $u \in C^0(X; \R)^K$. 

We should show that the same holds for any $u \in C^0(X; \R)$. 
By Lemma \ref{d_g is affine}, ${\mathbf d}_\eta$ is constant on the center. 
We observe that for any one-parameter subgroup $\mu \in N$ the slope of ${\bf d}(\mu(e^{-t}))$ is equivalent to the classical Futaki character. 
Since the character is defined on the reductive Lie algebra $\fk_\C$ which can be written as the direct sum of the center and the derived algebra, the slopes are nontrivial only on the center. Therefore ${\bf d}_\eta$ is actually constant on whole $K_\C$. 
Thus the measure $\mu:= (d D)_\phi$ is $K$-invariant. It then follows that for any smooth function $v$ and $g \in K$ 
\begin{equation*}
\int_X v\mu =\int_X g_*(v\mu)
=\int_X ((g^{-1})^*v)g_*(\mu) 
=\int_X ((g^{-1})^*v)\mu.  
\end{equation*}
Integrating against the Haar measure we have 
\begin{equation*}
\int_X v \mu = \int_X u\mu =0  
\end{equation*}
so that $\mu=0$ as desired. 

Conversely, if $\phi \in \cE^1(X, -K_X)^K$ is a weak solution, convexity of $D_g$ implies that $\phi$ is a minimizer.  
\end{proof} 

In the next subsection a refinement of the latter argument will show the uniqueness of the weak Mabuchi soliton. 

Let us discuss about the coercivity of the modified D-energy. 
From now on we rather start from the extremal one-parameter subgroup $\eta$. 
Let 
\begin{equation}
\Aut(X, \eta) := \bigg\{ g \in \Aut(X, -K_X): \eta(\tau) g=g\eta(\tau) \text{ for all } \tau \in \G_m. \bigg\}  
\end{equation}
The identity component is denoted by $\Aut^0(X, \eta)$.  
We take afresh $T=C(\Aut^0(X, \eta))$ as the center of the automorphisms commuting with $\eta$. 
Moreover, we entirely consider a maximal compact subgroup $K$ containing $S$. It clearly commutes with the center. 
We set 
\begin{equation}
J_T(\phi) := \inf_{\s \in T} J (\s^*\phi). 
\end{equation}

\begin{dfn}\label{coercivity of D_g}
Let $T=C(\Aut^0(X, \eta))$ and $K$ be a maximal compact subgroup of $\Aut^0(X, \eta)$, which contains the compact part of $T$. 
We say that the modified D-energy is coercive if there exists a positive constants $\e, C$ such that 
\begin{equation*}
D_\eta (\phi) \geq \e J_T (\phi) -C 
\end{equation*}
holds for every invariant metric $\phi \in \cH(X, -K_X)^K$. 
\end{dfn} 


By the standard argument we may obtain the weak minimizer from the coercivity. 
Actually for a minimizing sequence $\phi_j$, we have $\s_j \in T$ by the coercivity such that $\s_j^*\phi_j$ is contained in the sublevel set $\{ J_\eta \leq C\}$. 
Since $\{ E \geq -C\} $ is weakly compact, we obtain a weakly convergent subsequence $\s_j^*\phi_j \to \phi$ in $\cE^1(X, -K_X)^K$. 
From Lemma \ref{d_g is affine} the map $\s \mapsto D_\eta(\s^*\phi_j)$ is constant. That is, $D_\eta$ must be $T$-invariant. 
Especially $D_\eta(\s_j^*\phi_j)=D_\eta(\phi_j)$. 
The lower-semicontinuity concludes that $\phi$ is a minimizer of $D_\eta$. 
Theorem \ref{minimizer is a weak solution} concludes that the obtained minimizer is a weak Mabuchi soliton. 
It is indeed a smooth Mabuchi soliton, by Theorem \ref{regularity of Mabuchi soliton}. 
Our goal in this subsection is: 

\begin{thm}[\cite{LZ17}]\label{metric vs coercivity}
A Fano manifold $X$ admits a Mabuchi soliton if and only if $m_X >0$, $\Aut^0(X, \eta)$ is reductive, and the modified D-energy is coercive.   
\end{thm} 

We have already explained that the coercivity implies the existence of the metric. 
The converse direction is based on \cite{DR15}. 
For the Mabuchi soliton \cite{LZ17} already obtained a result adopting a different definition of the coercivity. 
Since the difference of the formulation is subtle point for the equivariant setting, we briefly sketch the proof in our framework.  

\begin{proof} 
Let $\phi$ be the Mabuchi soliton.
Trivially $K \subset \Aut(M, \phi)$ so the maximality implies $K =\Aut(M, \phi)$. 
Let $G=\Aut^0(X, \eta)$. 
By Corollary \ref{reductivity} $G=K_\C$ is reductive. 
We consider the normalizer and the centralizer 
\begin{align*}
&N_K(G):= \{ g \in G: gkg^{-1} \subset K\}, \\
&C_K(G):= \{ g \in G: gkg^{-1}=k \text{ for every } k\in K\}. 
\end{align*} 
We first observe $C_K(G) =C(G)$. 
Indeed any $t \in C_K(G)$ we have the map $\tau \colon G \to G$ defined by $ \tau (g) =tgt^{-1}$ and this is identical for $g \in K$. 
By Corollary \ref{reductivity} it implies that $\tau$ is identical on $G$. 

From the general theory of Lie groups we know that $N_K(G)/KC_K(G)$ is finite. 
Let us show $N_K(G)=KC(G)$ in our situation. 
Since $N_K(G)/KC_K(G)$ is finite, we may write $N_K(G) =K' C_K(G)=K'C(G)$ for some maximal compact subgroup $K'$. 
By construction $K \subset K' $ so the maximality of $K$ implies $N_K(G) =KC(G)$.  

We now check that $T=C(G)$ acts transitively on the smooth Mabuchi solitons. 
By Theorem \ref{uniqueness}, for two Mabuchi solitons $\phi$ and $\phi'$ we have $f \in \Aut^0(X, \eta)$ such that $f^*\phi = \phi'$. 
Since $K = \Aut(M, \phi)=\Aut(M, \phi')$ as we have already observed, it follows $f^{-1}Kf \subset K$. Namely, $f \in N_K(G)=KC(G)$. 

The above transitivity of $T$ and the regularity of weak minimizers (Theorem \ref{regularity of Mabuchi soliton}), we may apply \cite{DR15}, Theorem $3.4$ (with $\cR=\cH(X, -K_X)^K$, $G=C(\Aut^0(X, \eta))$ there) so that have constants $\e, C$ and 
\begin{equation}
D_\eta(\phi) \geq \e \inf_{\s \in T} J(\s^*\phi) -C 
\end{equation} 
for every $\phi \in \cH(X, -K_X)^K$. 
\end{proof} 

\subsection{Uniqueness of Mabuchi soliton}\label{uniqueness of Mabuchi soliton}

We shall first check the regularity.  
In \cite{LZ17} the corresponding step is carried out by the continuity method assuming the coercivity. We here introduce a direct argument. 

\begin{thm}\label{regularity of Mabuchi soliton} 
Assume $m_X = \inf_X (1+h_\eta)$ is strictly positive. 
Then the weak Mabuchi soliton of (\ref{weak Mabuchi soliton}) is actually smooth. 
\end{thm} 

\begin{proof} 
Since $\phi$ has finite Monge-Amp\`ere energy it has zero Lelong number (see \cite{GZ17}, Exercise 10.7). 
By the uniform version of Skoda's integrability theorem (\cite{GZ17}, Theorem 8.11), $\mu_\phi$ has $L^p$-density for any $p>1$. 
Noting (\ref{bound of g-MA}) and applying the viscosity theory: \cite{EGZ11}, Theorem C to $\MA_g(\phi) = \mu_\phi$, we deduce that $\phi$ is continuous. 
We may further show $\phi$ is $C^\infty$ essentially using Yau's $C^2$-estimate. For example one can apply the idea of \cite{ST09} Theorem 1 to the present setting. 
See \cite{ST19} for the detail exposition. 
\end{proof} 

For the uniqueness the fact $\phi \in L^\infty$ is important, because we need the following. 

\begin{prop}[\cite{Bern11}, Theorem 1.2]\label{holomorphic vector field}
Let $\phi^t$ be a weak geodesic which is uniformly bounded in the sense that $\abs{\phi^t -\phi_0} \leq C$. 
If the convex function $L(\phi^t)$ is affine, there exists a $f_t \in \Aut(X, -K_X)$ such that 
\begin{equation*}
f_t^* \omega_{\phi^t}=\omega_{\phi^0}. 
\end{equation*}
Moreover $f_t=\exp(-t\RE{v})$ for some holomorphic vector $v$ lifted to $-K_X$  such that $\IM v$ preserves $\omega_{\phi^t}$. 
\end{prop} 
\begin{rem}
By \cite{Berm16}, Proposition 3.3, we may further conclude $f_t^*\phi^t=\phi^0$. 
\end{rem}
For $\mu \in N_\R$ we denote by $\Aut(X, \mu)$ the group of bundle automorphisms preserving $\mu$.  
Set $\Aut(X, \phi)$ for a fiber metric $\phi$ in a similar manner. 

\begin{thm}\label{uniqueness} 
Let $(\omega_0, \eta_0)$ and $(\omega_1, \eta_1)$ be smooth Mabuchi solitons. 
Then there exists some $f \in \Aut^0(X, -K_X)$ such that 
\begin{equation*}
f^*\omega_1 = \omega_0, ~~~f^*\eta_1 =\eta_0. 
\end{equation*} 
If $\eta_0=\eta_1$ we have $f \in \Aut^0(X, \eta_1)$ and  
one can further take $f$ generated by the imaginary part of $\Aut(X, \phi_1)_\C$.  
\end{thm} 

\begin{proof}
First we consider the case $\eta = \eta_0 = \eta_1$ contained in the torus $T$.  
Take potentials $\phi^0, \phi^1$ of $\omega_0, \omega_1$ 
and bounded geodesic $\phi^t$ ($t \in [0, 1]$). 
Since $\phi^0$, $\phi^1$ are minimizers the convex function $D_\eta(\phi^t)$ should be affine. 
In particular $L(\phi^t)$ is affine. 
We may apply Theorem \ref{holomorphic vector field} so that $f_t^*\phi^t=\phi^0$. 
Observe that $\phi^t=(f_t^{-1})^*\phi^0$ is a weak Mabuchi soliton, since it is a minimizer of $D_\eta$. 
If we take the extremal vector field $v$ generating $\eta$ and set $w:=(f_t)_*v-v$, it follows $L_w\omega_0=0$ and hence $dd^c h_w=0$. 
That is, $f_t$ preserves $\eta$. 

When $\eta_0\neq \eta_1$ noting that the maximal tori are conjugate to each other we may take some $f$ so that $\eta_1=f^*\eta_0$, by the uniqueness of the extremal vector field. 
\end{proof} 

The uniqueness argument is closely related to the reductivity result. 

\begin{cor}\label{reductivity}
If a Fano manifold $X$ admits a Mabuchi soliton $(\omega_\phi, \eta)$ we have 
\begin{equation*}
\Aut^0(X, \eta)=\Aut(X, \phi)_\C. 
\end{equation*}
That is, $\Aut^0(X, \eta)$ is a complexification of the compact Lie group $\Aut(X, \phi)$. 
\end{cor} 

\begin{proof}
From $h_\eta =1-e^\rho$ we know $\Aut(X, \phi) \subset \Aut^0(X, \eta)$. 
If we take $g \in \Aut^0(X, \eta)$, $g^*\phi$ is Mabuchi soliton hence some $f \in \Aut(X, \phi)_\C$ satisfies $f^*\phi=g^*\phi $. It follows $g =  (g \circ f^{-1}) \circ f \in \Aut^0(X, \phi)_\C$.  
\end{proof} 

Mabuchi first showed Theorem \ref{uniqueness} using the inverse-continuity method of \cite{BM85}. 
In \cite{Mab03} Corollary \ref{reductivity} is also proved by the twisted Laplacian calculas similarly to \cite{Mat57}. 
As it was shown in \cite{Nak18}, one can also derive Corollary \ref{reductivity} directly from the second variation of the Ricci-Calabi functional. 
The present proofs are based on the idea of \cite{Bern11} for the K\"ahler-Einstein metric. 
A virtue of this idea more directly links reductivity to the uniqueness. 

\subsection{Thermodynamical formalism and modified K-energy}\label{thermodynamical formalism and modified K-energy}

In this final part of the section, following the thermodynamical formalism of \cite{Berm13} and its modified version in \cite{BWN14}, we introduce the modified K-energy in terms of D-energy. 

Recall for two probability measures $\mu, \nu$ the relative entropy is defined to be 
\begin{equation}
H(\mu \vert \nu) = \int_X \log \bigg[\frac{d\mu}{d\nu}\bigg] d\mu. 
\end{equation}
Its relation with D-energy is based on the Legendre transformation formula:  
\begin{equation}\label{L transforms into H}
H(\mu \vert \mu_0) 
= \sup_{f \in C^0(X; \R)} \bigg[ \int_X f d\mu  -\log \int_X e^f d\mu_0 \bigg]. 
\end{equation}

\begin{dfn}\label{modified K-energy}
Fix a reference $\phi_0 \in \cH (X, -K_X)^K$ and $\mu_0:=\mu_{\phi_0}$.  
Let $g \colon P \to \R$ be a positive continuous function on the moment polytope. 
For $\mu$ with finite $E^*(\mu)$ we define the free energy 
\begin{equation}
F(\mu) := 
H(\mu \vert \mu_0) - E^*(\mu). 
\end{equation} 
For $\phi \in \cE^1(X, -K_X)^K$ we define the K-energy as $M(\phi):=F(\MA(\phi))$ and the modified K-energy as 
\begin{equation}\label{Chen type formula}
M_g(\phi) :=  H(\MA(\phi) \vert \mu_0) -E_g(\phi)+ \int_X (\phi-\phi_0) \MA(\phi). 
\end{equation} 
\end{dfn} 

In \cite{BWN14} the K\"ahler-Ricci soliton case (\ref{choice of g for KRS}) was discussed. In this case $M_g$ is equivalent to the energy introduced by \cite{TZ02} for smooth metrics. The treatment is valid for arbitrary $g$ including Mabuchi soliton case. 
From the definition it follows 
\begin{align*}
F(\mu) 
&= H(\mu \vert \mu_0) -E^*(\mu)  \\ 
&= \sup_{f}  \bigg[ \int_X f d\mu  -\log \int_X e^f d\mu_0 \bigg] 
- \sup_\phi \bigg[ E(\phi) -\int_X (\phi -\phi_0) d\mu \bigg] 
\end{align*}
and the second supremum is attained by the weak solution of $\MA(\phi) =\mu$. 
Therefore we obtain $M_{1}=M$, which is analogues to the Chen-Tian formula (\cite{Chen00b}). 
In particular the $g=1$ case gives the original definition of K-energy. 

\begin{rem}\label{about M_g}
\begin{itemize}
\item[$(1)$]
It seems also natural to consider the functional $F(\MA_g(\phi))$ but we adopt the above $M_g$. 
This is mainly because the convexity property of $F(\MA_g(\phi))$ is unclear. 
\item[$(2)$]
There are already several functionals called modified K-energy in the literatures, which are defined mainly to characterize the extremal K\"ahler metric. 
For example the functional 
$M':= M -E_\eta$
gives one such candidate. 
An extremal metric might not be a Mabuchi soliton unless it is K\"ahler-Einstein. 
Since $M \geq D $ clearly implies $M' \geq D_\eta $, if $D_\eta$ is coercive so does $M'$. 
It follows that if $X$ admits a Mabuchi soliton it also has an extremal K\"ahler metric. The converse does not holds. For example, $X=\P(\cO_{\P^2}\oplus \cO_{\P^2}(2))$ admits an extremal K\"ahler metric but not Mabuchi soliton. See \cite{NSY17} for the detail. 
\end{itemize}
\end{rem} 

First of all, we have the following convexity property of $M_g$. 

\begin{thm}[A slight generalization of \cite{BB17}, \cite{BDL15}]\label{convexity of M_g}
Assume $g>0$. 
The modified K-energy $M_g$ is convex along any psh geodesic in $\cE^1(X, -K_X)^K$. 
\end{thm} 

\begin{proof}
The result for $g=1$ was first proved by \cite{BB17}, assuming that the psh geodesic has bounded Laplacian. 
It was extended to arbitrary psh geodesic in $\cE^1$ by \cite{BDL15}. 
As we have already observed that $E_g$ is geodesically affine, totally the same argument works for general choice of $g>0$. 
\end{proof} 

The present choice of $M_g$ shares the same minimizer with $D$. 

\begin{thm}[\cite{BWN14}, Proposition $3.2$]
We have $M_g \geq D_g$ on $\cE^1(X, -K_X)^K$ and a metric $\phi$ attains the equality iff it is a weak Mabuchi soliton.  
The modified K-energy $M_g$ is lower bounded iff $D_g$ is. 
In this case the infimums of the both functionals coincide.  
\end{thm} 

\begin{proof}
For the reader's convenience we give the proof. 
The one-side inequality in (\ref{L transforms into H}) is a simple consequence of Jensen's inequality and actually holds for lower-semicontinuous function of the form $f=-(\phi -\phi_0)$. 
It immediately shows $M_g \geq D_g$. 

On the other hand, the supremum is attained by the solution $f$ of 
\begin{equation*}
\frac{e^f d\mu_0}{\int_X e^f d\mu_0} = d\mu.  
\end{equation*}
Consequently, $M_g(\phi) = D_g(\phi)$ iff $\phi$ is a weak Mabuchi soliton. 

It remains to show $\inf_\phi M_g =\inf_\phi D_g \in \R \cup \{ -\infty\}$. 
Set $m:= \inf_\phi M_g$. 
By the properness result \cite{BBEGZ16}, Theorem $2.18$, we have 
\begin{equation}
H(\mu \vert \mu_0) \geq \a E^*(\mu) -C 
\end{equation}
for any $\a $ smaller than Tian's $\a$-invariant. 
In particular, $H(\mu \vert \mu_0) < \infty$ implies that $\mu$ has finite energy so that some $\phi\in \cE^1(X, -K_X)^K$ solves $\MA(\phi)=\mu$. 
Substitution to (\ref{Chen type formula}) yields 
\begin{equation*}
H(\mu \vert \mu_0) \geq m+ E_g(\phi) -\int_X (\phi-\phi_0)\MA(\phi). 
\end{equation*}  
Since the infimum of the inversion formula 
\begin{equation*}
L(\phi) = \inf_\mu \bigg[ H(\mu \vert \mu_0) + \int_X (\phi- \phi_0) d\mu \bigg] 
\end{equation*}
is attained by $\mu= \mu_\phi$, it follows  
$D_g \geq	 m$. 
\end{proof} 

Totally in the same manner we observe that 
the coercivity $M_g(\phi) \geq \e \inf_{\s \in T}J_g(\s^*\phi) -C$ holds on $\cE^1(X, -K_X)^K$ iff $D_g\geq \e \inf_{\s \in T} J_g(\s^*\phi) -C$ on $\cE^1(X, -K_X)^K$.


\section{Relative uniform D-stability} 

Bearing in mind of the last section, we introduce the algebraic (non-Archimedean) counterpart of  modified energies and define the appropriate notion of stability which should characterize the existence of Mabuchi soliton. 
We first recall the notion of D-stability introduced by \cite{Berm16}. 
The terminology here is due to our previous work \cite{BHJ15}, \cite{BHJ19}. 

\subsection{Uniform D-stability}\label{uniform D-stability}

We first require that a test configuration $\pi\colon (\cX, \cL)\to \A^1$ of a polarized manifold $(X, L)$ is a family of polarized schemes defined over the affine line $\A^1$. 
In fact we may and should allow $\cL$ to be only relatively semiample and $\Q$-Cartier divisor. 
Further, $\cX$ is normal variety and endowed with a lifted $\G_m$-action $\lambda\colon \G_m \to \Aut(\cX, \cL)$ such that the projection $\pi$ is equivariant for $\lambda$ and the standard $\G_m$-action to $\A^1$. 
The datum includes the isomorphism 
\begin{equation*}
 \pi^{-1}(\A^1\setminus \{0\}) \simeq X\times (\A^1 \setminus \{0 \}) 
\end{equation*}
which sends the line bundle $\cL$ equivariantly to $L_{\A^1}=p_1^*L$. 
Although we are concerned with the case $L=-K_X$, $\cL$ is still not equivalent to $-K_{\cX/\P^1}$. 
Since we assumed $\cX$ is normal, $K_{\cX}$ is at least well-defined as a Weil divisor, however, it is even not the line bundle in general. 
Note that some literatures consider the family over the projective line $\P^1$. 
This is equivalent to our setting because one can always obtain the unique compactified family $(\bar{\cX}, \bar{\cL}) \to \P^1$ which is trivial around $\infty \in \P^1$. 

\begin{exam}
Every one-parameter subgroup $\mu \in N$ defines a product family $X_{\A^1}=X \times \A^1$ endowed with the non-trivial action: $\lambda(\s)(x, \tau):= (\mu(\s)x, \s\tau)$ for $\s \in \G_m$. 
We call it a product configuration generated by $\mu$. 
Therefore, test configuration can be seen as a far generalization of one-parameter subgroup. 
Note that the compactified family is no longer a product space. 
For example, the product space $\P^1\times \A^1$ endowed with a non-trivial $\G_m$-action is compactified to the Hirzebruch surface $\P(\cO_{\P^1}\oplus \cO_{\P^1}(d))$. 
\end{exam} 

After the compactification $(\bar{\cX}, \bar{\cL})$ we may take the intersection number \eg $\bar{\cL}^{n+1}$. 
Since we assumed $\cX$ is normal, $K_{\cX}$ is at least well-defined as a Weil divisor and $K_{\bar{\cX}}\bar{\cL}^n$ essentially gives the famous Donaldson-Futaki invariant, or equivalently, non-Archimedean K-energy $M^\NA(\cX, \cL)$ introduced in \cite{BHJ15}. 
This is a natural generalization of (\ref{L-Futaki invariant}) to arbitrary test configurations. 

Let us define the non-Archimedean D-energy. 
Recall that for given divisors $\cB, \cD$ the log-canonical threshold 
$\lct_{(\bar{\cX}, \cB)}(\cD)$ is defined to be the supremum of $c \in \R$ such that  
the log pair $(\bar{\cX}, \cB+c\cD)$ has at worst log canonical singularities. 
Choosing the boundary divisor $\cB$ linearly equivalent to $-K_{\bar{\cX}/\P^1}-\bar{\cL}$, the quantity reflects the positivity of the canonical divisor.  
Notice that in this choice the log-canonical divisor $K_{\bar{\cX}} + \cB \sim_\Q -\bar{\cL}+\pi^*K_{\P^1}$ is $\Q$-Cartier so that the log discrepancies and log canonical singularities are well defined for any $(\cX, \cL)$. 

\begin{dfn}\label{definition of D^NA}
For a test configuration $\pi \colon (\cX, \cL) \to \A^1$ we define 
\begin{equation*}
L^\NA(\cX, \cL) := 
\lct_{(\bar{\cX}, \cB)} (\cX_0) -1, ~~~
E^\NA(\cX, \cL) := 
\frac{\bar{\cL}^{n+1}}{(n+1)V},  
\end{equation*}
where $\cX_0$ is the scheme theoretic central fiber and the boundary divisor is chosen $\cB \sim_\Q -K_{\bar{\cX}/\P^1}-\bar{\cL}$. 
We say a Fano manifold $X$ is D-semistable if the non-Archimedean D-energy 
\begin{equation*}
D^\NA(\cX, \cL):= L^\NA(\cX, \cL) -E^\NA(\cX, \cL)
\end{equation*}
is semipositive for all test configurations. 
\end{dfn} 

If $(\cX, \cL)$ is the product test configuration generated by $\mu \in N$, we write $D^\NA(\cX, \cL)$ as $D^\NA(\mu)$. 
It is known to be equivalent to the Futaki character (\ref{Futaki invariant}), \ie   
\begin{equation}
D^\NA(\mu) =F(\mu)  
\end{equation} 
holds for every $\mu \in N$. 
One can also define D-stability and D-polystability of a Fano manifold. 
See \cite{Berm16}, \cite{BHJ15}, and \cite{Fuj16} for the detail treatment. 
The uniform version is more important for us. 
We say that a test configuration $(\cX', \cL')$ is a pull-back of $(\cX, \cL)$ if a birational equivariant morphism $f\colon \cX' \to \cX$ yields $\cL'=f^*\cL$. 
Two test configurations are called {\em equivalent} if they admit a common pull-back. 
It is easy to see that the above invariants have the same values for the equivalent test configurations. 
More substantially, by \cite{BHJ15}, any test configurations can be seen as a fiber metric of the Berkovich analytification $(\cX^\NA, \cL^\NA)$, evaluating each valuation on the central fiber. 
The equivalence of test configurations are precisely of corresponding non-Archimedean fiber metrics. 
The above $L^\NA$ and $E^\NA$ are actually functionals defined on these non-Archimedean fiber metrics. 
From this reason, taking a pull-back we may assume a domination $\rho \colon  (\cX, \cL) \to X_{\A^1}$ to the product family endowed with a possibly non-trivial action. 
By the projection formula the following definition is actually independent of $\rho$. 

\begin{dfn}\label{definition of J^NA}
We define 
\begin{equation*} 
L_0^\NA (\cX, \cL) := V^{-1}(\rho^*L_{\A^1})\bar{\cL}^n. 
\end{equation*}
and the non-Archimedean counterpart of Aubin's J-functional as 
$J^\NA(\cX, \cL):= L_0^\NA (\cX, \cL)- E^\NA(\cX, \cL)$. 
A Fano manifold is called uniformly D-stable if there exists a constant $\e>0$ such that 
\begin{equation*}
D^\NA(\cX, \cL) \geq \e J^\NA(\cX, \cL) 
\end{equation*}
holds for all test configurations. 
\end{dfn} 

Let us illustrate a key relation between the functionals $E, J, D$ and their non-Archimedean version. 
It explains that test configuration gives the algebraic formulation of the geodesic ray on $\cH$. 
In the sequel we denote the fiber of $\tau \in \A^1$ by $\cX_\tau$ and the restricted line bundle by $\cL_{\tau}$. 
As well, for the unit disk $\Delta$ we set $\cX_{\Delta}:=\pi^{-1}(\Delta)$ and $\cL_\Delta:= \cL\vert_{\cX_{\Delta}}$. 
For the punctured disk $\Delta^*=\Delta \setminus \{0\}$ we have the isomorphism $\cX_{\Delta^*} \simeq X \times \Delta^* $ so that identify a point of $\cX_{\Delta^*}$ with $(x, \tau)$. 
Let $\Phi$ be a smooth fiber metric of $\cL_\Delta$, having the semipositive curvarture. 
It defines the ray 
\begin{equation}\label{compatibility of ray}
\phi^t(x) = \Phi(\lambda(e^{-t})(x, 1))
\end{equation}
so that $\phi^t$ for each $t \in [0, \infty)$ defines a fiber metric of $L$, having the semipositive curvature. 
This type of ray is said to be compatible with the test configuration. 
Any two metrics defines the same asymptotic because the difference of the associated rays is bounded uniformly in $t$. 
The following type of results is predicted in the origination of K-stability and proved for arbitrary test configurations in \cite{Berm16}, \cite{BHJ19}. 

\begin{thm}\label{slope of D}
Let $F\colon \cH \to \R$ be a functional either $E, J,$ or $D$. 
For a test configuration and a ray $\phi^t$ compatible with $(\cX, \cL)$ we have 
\begin{equation*}
F^\NA(\cX, \cL) = \lim_{t \to \infty} \frac{F(\phi^t)}{t}. 
\end{equation*} 
\end{thm} 

The above formula is indeed true for non-smooth but bounded $\Phi$ for which the semipositivity $dd^c\Phi \geq0$ holds in the sense of current. 
In particular the same result holds for the associated psh geodesic ray $\phi^t$ which is characterized by the degenerate Monge-Amp\`ere equation 
\begin{equation}\label{associated weak geodesic ray}
(dd^c_{\tau, x}\Phi)^{n+1} =0 
\end{equation} 
on $\cX_\Delta$. 
Given a smooth boundary value $\phi^0$, the bounded solution uniquely exists. 
For example \cite{Berm16} gives the solution in terms of the Peron-Bremermann type envelope. 
After \cite{CTW18} it is known to have the best-possible $C^{1, 1}$-regularity. 
We have already emphasized that the consideration of weak geodesic is necessary for the variational approach. 

Regarding Theorem \ref{slope of D}, as in \cite{DL18} we introduce for psh geodesic ray $\phi^t$ the {\em radial} energy 
\begin{equation}
\hat{F}(\Phi) := \lim_{t \to \infty} \frac{F(\phi^t)}{t}, 
\end{equation} 
which is well-defined by the convexity property. 

As a consequence of Theorem \ref{slope of D}, the coercivity of the D-energy implies that $X$ is uniformly D-stable. 
The heart of \cite{BBJ15} is showing the converse direction. 
It is known that the uniform stability implies that the automorphism group is finite (see \cite{BHJ19} for the analytic discussion and the purely algebraic proof \cite{BX18}).  
What we are going to discuss suggests the one of the treatment for general automorphism groups. 

\subsection{Associated concave function and Duistermaat-Heckman measure}\label{associated concave functiob and Duistermaat-Heckman measure}

We continuously fix an extremal one-parameter subgroup $\eta$ and a torus $T \subset \Aut^0(X, \eta)$. As in the previous subsections we denote the lattice of one-parameter subgroups by $N$ and the dual by $M$. 
Let $P \subset M_\R$ be the moment polytope of the maximal torus and 
\begin{equation*}
m_\phi\colon X \to P 
\end{equation*} 
the moment map. 
Recall that the Duistermaat-Heckman measure is the push-forward 
\begin{equation}
\DH_T:= (m_\phi)_*(V^{-1}\omega^n) 
\end{equation} 
which is also independent of the metric. 
Let us first give an algebraic definition of $\DH_T$. 
Any $\mu \in N_\R$ is identified with the affine function $G_\mu(x):=\langle \mu, x \rangle$ on $M_\R$ hence we may integrate by $\DH_T$.  
Let $k \in \N$ and $\mu_1, \dots, \mu_{N_k}$ be the weight of the $\G_m$-action on $H^0(X, kL)$, induced by $\mu$. 
For any $p \geq 1$, the equivariant Riemann-Roch formula implies 
\begin{equation}
\int_P G_\mu^p(x)  \DH_T 
= \lim_{k\to \infty} \frac{1}{k^pN_k}\sum \mu_i^p. 
\end{equation} 
If set $\DH_\mu := (h_\mu)_*(V^{-1}\omega^n)= (G_\mu)_*\DH_T $, by the Hausdorff moment theorem we obtain 
the convergence of the measures on $\R$: 
\begin{equation}\label{weight description of DH_T}
\DH_\mu = \lim_{k \to \infty} \frac{1}{N_k}\sum \delta_{\frac{\mu_i}{k}}. 
\end{equation}
A simple argument checks that $P$ is the closed convex hull of the set 
\begin{equation}
\bigg\{ \frac{\chi}{k}\in M_\Q: \chi \in M, s_\chi \in H^0(X, kL) ~\text{with}~ \s \cdot s_\chi = \chi(\s) s_\chi \bigg\}  
\end{equation}
and that 
\begin{equation}
\DH_T = \lim_{k \to \infty} \frac{1}{N_k}\sum \d_{\frac{\chi}{k}}, 
\end{equation}
where $\chi \in M$ runs for all $s_\chi \in H^0(X, kL)$. 
As a consequence 
\begin{prop}\label{m_X formula}
For every $\mu \in N_\Q$ we have 
\begin{equation*}
m_X := \inf_P G_{1+\mu} = \lim_{k \to \infty} \min \frac{k+\mu_i}{k}  
= \inf_X (1+h_\mu). 
\end{equation*}
In particular $\inf_X (1+h_\mu)$ is independent of the metric. 
\end{prop}

More generally, a test configuration defines a concave function $G_{(\cX, \cL)}$ on $P$. 
For the purpose it is convenient to describe the test configuration in terms of the filtration. 
\begin{dfn} 
Let $(\cX, \cL)$ be a test configuration. 
Given $s \in H^0(X, L)$, we have a rational section 
\begin{equation*}
\bar{s}(x, \tau) = \lambda(\tau)\cdot s(\lambda(\tau^{-1})(x, \tau))
\end{equation*}
of $\cL$. 
Considering how extent $\bar{s}$ is holomorphic we obtain a filtration of the section ring, which fully recovers the test configuration. 
For each $\lambda \in \R$ we set 
\begin{equation}
F^{\lambda}H^0(X, kL) := 
\{ s \in H^0(X, kL): \tau^{-\lceil \lambda \rceil } \bar{s} \in H^0(\cX, k\cL) \}. 
\end{equation}
\end{dfn} 
We may easily show that the filtration is monotone, left-continuous, and multiplicative in both $k$ and $\lambda$.  
By \cite{BHJ15} lemma 2.14, $\lambda$-weightspace of the induced action to $H^0(\cX_0, k\cL_0)$ is given by 
\begin{equation}\label{filtration to central fiber}
H^0(\cX_0, k\cL_0)_\lambda
\simeq F^{\lambda}H^0(X, kL)/F^{\lambda+1}H^0(X, kL). 
\end{equation}
The identification also follows from Proposition \ref{T-equivariant trivialization} below. 
A non-trivial fact proved in \cite{PS07} is the linearly boundedness. Namely there exists a constant $C>0$ such that 
\begin{equation}\label{linearly bounded}
F^{kt}H^0(X, kL) = \{ 0\} ~~~~~~(\text{{\em resp}. }  H^0(X, kL))
\end{equation} 
for any $t>C$ ({\em resp}. $t <-C$) and $k \geq 1$. 
It is equivalent to say: $\abs{\lambda} \leq Ck$ for the induced $\G_m$-action. 

Imitating \cite{WN12}, we construct a concave function from the filtration. 

\begin{dfn}
For each $t \in \R$ we define $P^t$ as the closed convex hull of the set 
\begin{equation*}
\bigg\{ \frac{\chi}{k}\in M_\Q: \chi \in M, s_\chi \in F^{kt}H^0(X, kL) ~\text{with}~ \s \cdot s_\chi = \chi(\s) s_\chi\bigg\}.  
\end{equation*} 
The associated concave function is then defined as 
\begin{equation*}
G_{(\cX, \cL)}(x) 
:= \sup \{ t \in \R: x \in P^t \}. 
\end{equation*} 
\end{dfn} 
It is easy to check $G_{(\cX, \cL)}=G_\mu$ when $(\cX, \cL)$ is the product configuration generated by $\mu \in N$. Indeed from the definitions we compute 
\begin{align*}
\tau^{-kt}\overline{s_\chi}(x, \tau)
&=\tau^{-kt} \mu(\tau)\cdot s_\chi(\mu(\tau^{-1})x) \\
&=\tau^{-kt}(\mu(\tau)\cdot s_\chi)(x) \\
&= \tau^{-kt+\langle \mu, \chi \rangle}s_\chi (x). 
\end{align*}

In terms of the associated weak geodesic ray (\ref{associated weak geodesic ray}) we may extend (\ref{weight description of DH_T}) to $T$-equivariant test configurations. 
Notice that the weak geodesic ray $\phi^t$ has $C^{1, 1}$-regularity and the right-derivative 
\begin{equation}
\dot{\phi}^0
:= \inf_{t>0} \frac{\phi^t -\phi^0}{t} 
\end{equation}
pointwisely defined is hence a bounded function. 
This fact reflects the linearly boundedness (\ref{linearly bounded}).  

\begin{thm}
Let $(\cX, \cL)$ be a $T$-equivariant test configuration. 
For each $k \in \N$ $\lambda_1, \dots, \lambda_{N_k}$ denote the weights of the induced $\G_m$-action to $H^0(\cX_0, k\cL_0)$. 
The push-forward 
\begin{equation*}
\DH_{(\cX, \cL)} := \dot{\phi}^0_*(V^{-1} \omega^n) 
\end{equation*} 
defines a probability measure on $\R$, which is independent of the metric. 
Moreover, it is equivalent to 
\begin{equation*}
(G_{(\cX, \cL)})_*\DH_T = \lim_{k \to \infty} \frac{1}{N_k}\sum \delta_{\frac{\lambda_i}{k}}. 
\end{equation*}
\end{thm} 

\begin{proof}
The identity $\DH_{(\cX, \cL)}=\lim_{k \to \infty} \frac{1}{N_k}\sum \delta_{\frac{\lambda_i}{k}}$ is shown in \cite{His16a}. 
For any $p \geq 1$ we have 
\begin{align*}
\int_\R t^p (G_{(\cX, \cL)})_* \sum \d_{\frac{\chi}{k}} 
= \int_P G_{(\cX, \cL)}^p \sum \d_{\frac{\chi}{k}} 
=\sum G_{(\cX, \cL)}^p(\frac{\chi}{k}), 
\end{align*}
where the summation is for all $s_\chi \in H^0(X, kL)$. 
In view of the Hausdorff moment theorem it remains to show 
\begin{equation*}
\sum G_{(\cX, \cL)}^p(\frac{\chi}{k}) = \sum (\frac{\lambda_i}{k})^p. 
\end{equation*}
We fix $\chi$ and by the linearly boundedness (\ref{linearly bounded}) take the largest $t$ such that $s_\chi \in F^{\lceil kt \rceil}H^0(X, kL)$ but $s_\chi \notin F^{\lceil kt \rceil+1}H^0(X, kL)$. 
From (\ref{filtration to central fiber}) such $kt$ one-to-one corresponds to $\lambda_i$ so we complete the proof.  
\end{proof} 

\begin{rem}
In \cite{WN12}, the associated concave function on the {\em Okounkov body} is in fact defined for any possibly non-equivariant test configuration. 
If $X$ is toric polarized manifold the Okounkov body and the associated concave function is equivalent to the present construction. 
In our setting $G_{(\cX, \cL)}$ is to be a piecewise-linear function.
\end{rem} 

We conclude this subsection describing the invariants $E^\NA, J^\NA$ in terms of the $\G_m$-action. 

\begin{prop}[\cite{BHJ15} Proposition $7.8$, Theorem 5.16]\label{E in terms of action}
For any test configuration, the non-Archimedean Monge-Amp\`ere energy satisfies 
\begin{align*}
E^\NA(\cX, \cL) = \int_\R t \DH_{(\cX, \cL)} = \lim_{k\to \infty} \frac{1}{N_k}\sum_{i=1}^{N_k} \frac{\lambda_k}{k}. 
\end{align*} 
The functional $L_0$ satisfies 
\begin{align*}
L_0^\NA(\cX, \cL) = \sup \supp \DH_{(\cX, \cL)} = \lim_{k\to \infty} \max_i \frac{\lambda_i}{k}. 
\end{align*} 
Moreover, $\max_i \frac{\lambda_i}{k}$ is stable in $k$. 
More precisely, it is enough to take a sufficiently divisible $k$ so that $k\cL$ is globally generated. 
If there exists a domination $\rho\colon \cX \to X_{\A^1}$, 
let $E_0$ be the strict trandsform of $X\times\{0\}$ and $D:=\cL-\rho^*L_{\A^1}$ be the unique $\Q$-divisor supported on $\cX_0$.  
We then have 
\begin{align*}
L_0^\NA(\cX, \cL) =\ord_{E_0}D.
\end{align*}  
\end{prop} 

\subsection{Relative setting}\label{relative setting}

Let us discuss the relative stability. 
From now on we fix $\eta$, $G=\Aut(X, \eta)$, $T=C(G)$, and $K\subset G$ a maximal compact subgroup containing $S$.   
The identical one-parameter subgroup to $T$ is denoted by $1 \in N$. 
Test configurations are assumed to be $G$-equivariant, as the metrics were $K$-invariant. 

\begin{dfn}\label{T-equivariant test configuration} 
Let $G$ be a reductive algebraic group. A test configuration $(\cX, \cL)$ endowed with $\G_m \times G$-action is $G$-equivariant if it is compatible with the equipped $\G_m$-action on $(\cX, \cL)$ and the $G$-action on $(X, L)=(\cX_1, \cL_1)$. 
\end{dfn} 

Notice that we imposed the commutativity with $G$, on the $\G_m$-action. 
In particular $G$ acts on the central fiber $\cX_0$. 
The $G$-equivariance is not too much restrictive, as one can see \cite{DS16}, Theorem 1, in the K\"ahler-Einstein case. 
See also Example \ref{example of G-invariant ideals} below. 

\begin{exam}
Consider a $G$-invariant ideal $I \subset \cO_X$. 
Let $\rho\colon \cX\to X_{\A^1}$ be the normalization of the blow-up along the ideal $\cJ:=I\cO_{X_{\A^1}}+\tau\cO_{X_{\A^1}}$ and $E$ be the exceptional divisor.  
We take $\e>0$ and set $\cL:= \rho^*L_{\A^1} -\e E$.  
This typical test configuration called {\em deformation to the normal cone} is intensively studied in \cite{RT07}. 
Indeed one can show that $\cL$ is ample for every sufficiently small $\e$. 
If for example the support $V$ of $\cJ$ is smooth we may write $E=\P(N_{V/X} \oplus \cO_X)$ as the normal cone. 
The induced $\G_m$-action is trivial on the normal bundle $N_{V/X}$ and the is the simple multiplication on $\cO_X$. 
Since $\cJ$ is $G$-invariant, $(\cX, \cL)$ inherits the $G$-action so that $\rho$ is equivariant. 
In this construction we observe that the two actions actually commute to each other. 

Let us consider the case $G=\Aut(X, -K_X)$.  
For example, $\P^2$ does not have any $G$-invariant ideal. 
If $X$ is the one point blow-up of $\P^2$ any $G$-invariant ideal is supported on the exceptional divisor. 
We may check that the deformation to the normal cone prevent $X$ to be D-semistable. 
\end{exam} 

The starting point here is to take the inner product of such a test configuration with arbitrary one-parameter subgroups, extending the definition of \cite{FM95}. 
The equipped $\lambda\colon \G_m \to \Aut(\cX, \cL)$ induces the action to  $H^0(\cX_0, k\cL_0)$ for every $k\geq 1$.  
Since $\cX$ is normal and is a family over the curve, it is flat. 
It follows $H^0(\cX_0, k\cL_0) \simeq H^0(X, kL)$ for any sufficiently large $k$. 
In fact we may have a $G$-equivariant trivialization of the vector bundle $\pi_*(k\cL)$ over $\A^1$. 

\begin{prop}\label{T-equivariant trivialization}
The $G$-equivariant algebraic vector bundle $E=\pi_*(k\cL)$ on the affine line $\A^1$ is $G$-equivariantly isomorphic to $E_0 \times \A^1$. 
\end{prop} 
\begin{proof} 
For the case $G=\{\id \}$ we refer \cite{BHJ15} Proposition 1.3. 
Taking $M$ into the account the same argument works. 
Indeed from the commutativity of the first component $\G_m$ with the second $G$, $G$-action does not effect. 
Let $M_G$ be the lattice of weights and 
\begin{equation}
H^0(\A^1, E) = \bigoplus_{(\lambda, \chi) \in \Z \oplus M_G} H^0(\A^1, E)_{(\lambda, \chi)} 
\end{equation}
be the decomposition to the irreducible representations. 
Set $V:=E_1=H^0(X, kL)$, $V_\chi :=H^0(X, kL)_\chi$ and $F^\lambda V_\chi$ as the image of $H^0(\A^1, E)_{(\lambda, \chi)}$ under the restriction map $H^0(\A^1, E) \to E_1$. 
Definition \ref{T-equivariant test configuration} implies $F^\lambda V_\chi \subset V_\chi$. 
Since $\tau$ has weight $-1$ with respect to the $\G_m$-action on the base $\C$, multiplication by $\tau$ induces $F^{\lambda+1} V_\chi \subset F^{\lambda} V_\chi$. 
Since $F^\lambda V_\chi = V_\chi$ for $\lambda \ll 0$ and $V=\oplus_\chi V_\chi$, the above map sending $\sum \tau^{-\lambda} v_\lambda$ to $\sum v_\lambda$ is surjective. 
On the other hand, if $\sum \tau^{-\lambda} v_\lambda$ lies in the kernel, 
$w_\lambda:= -\sum_{\lambda' \geq \lambda} v_{\lambda'}$ in $F^\lambda V := \oplus_\chi F^\lambda V_\chi $ vanishes for $\lambda \ll 0$. 
Since $v_\lambda =w_{\lambda+1}-w_\lambda$, it means that $\sum \tau^{-\lambda} v_\lambda$ is in $(\tau-1)H^0(\A^1, E)$. 
Thus we have $H^0(\A^1, E)/(\tau-1)H^0(\A^1, E)\simeq V$ and the equivariant isomorphism 
\begin{equation*}
E\vert_{\A^1\setminus \{0\}} \simeq V \times \A^1\setminus \{0\}. 
\end{equation*} 
Similarly, by sending $\sum \tau^{-\lambda} v_\lambda$ to $v_{\lambda}$ modulo $F^{\lambda+1}V$ we may show 
\begin{equation*}
E_0 \simeq \bigoplus_{\lambda \in \Z} F^\lambda V/F^{\lambda +1}V. 
\end{equation*} 
It follows the equivariant isomorphism 
\begin{equation*}
H^0(X, E) \simeq \bigoplus_{\lambda \in \Z} \tau^{-\lambda} F^\lambda V. 
\end{equation*}
By choosing a basis compatible the filtration and $F^\lambda V = \oplus_\chi F^\lambda V_\chi $, we obtain a required equivariant trivialization. 
\end{proof} 

Now consider $G=\Aut^0(X, \eta)$. 
Because the test configuration is assumed to be $G$-equivariant,  
given $\mu \in N$ we may simultaneously diagonalize the two actions on $H^0(X, kL)$ and $H^0(\cX_0, k\cL_0)$ so that each weights $\lambda_i$ and $\mu_i$ are assigned for the common vectors under the equivariant trivialization. 
In the sequel we may take any such $\lambda_i$ and $\mu_i$.  

\begin{dfn}[\cite{His16b}]
Let $(\cX, \cL)$ be a $T$-equivariant test configuration. 
For any one-parameter subgroup $\mu \in N$ we have the limit 
\begin{equation*}
\langle (\cX, \cL), \mu \rangle
:= \lim_{k \to \infty} \frac{1}{k^2N_k} \sum_{i=1}^{N_k} \lambda_i \mu_i.  
\end{equation*} 
When $(\cX, \cL)$ is a product configuration generated by some $\lambda \in N$, we denote 
$\langle (\cX, \cL), \mu \rangle=: \langle \lambda, \mu \rangle$. 
\end{dfn} 

For the identical one-parameter subgroup $1 \in N$ 
we observe $\mu_i =k$ and hence Proposition \ref{E in terms of action} shows 
\begin{equation*}
\langle (\cX, \cL), 1 \rangle 
= \lim_{k \to \infty} \frac{1}{kN_k} \sum_{i=1}^{N_k} \lambda_i. 
\end{equation*} 
It is easy to check that homogeneity naturally extends the above inner product to $\mu \in N_\Q$. 
We may further extend the definition to $\mu \in N_\R$ by the following description. 

\begin{thm}[\cite{His16b}]\label{analytic formula for inner product of test configuration and 1-PS}
Let $(\cX, \cL)$ be a $T$-equivariant test configuration and $\mu \in N_\Q$. 
For the associated weak geodesic ray $\phi^t$ and the Hamilton function $h_\mu$ we have 
\begin{equation*}
\langle (\cX, \cL), \mu \rangle
= \frac{1}{V} \int_X \dot{\phi}^0 h_\mu \omega^n.  
\end{equation*}
\end{thm} 

In view of the above result, it is also natural to define 
\begin{equation}
\langle \Phi, \mu \rangle
:= \frac{1}{V} \int_X \dot{\phi}^0 h_\mu \omega^n 
\end{equation} 
for arbitraty psh geodesic $\Phi$ with $C^{1, 1}$-regularity. 
When $\Phi$ is associated with a test configuration the right-hand side gives the slope of the affine function $E_g(\phi^t)$.  

\begin{cor}\label{slope of modified E}
Let $(\cX, \cL)$ be a test configuration and $\phi^t$ the associated weak geodesic ray. 
For any $\mu \in N_\Q$ putting $g(x):= \langle x, \mu \rangle -\int_P \langle x, \mu \rangle\DH_T$ we have 
\begin{equation*}
\langle (\cX, \cL), \mu \rangle = \lim_{t \to \infty}  \frac{E_g(\phi^t)}{t}. 
\end{equation*} 
\end{cor}

Now we choose the extremal one-parameter subgroup $\eta \in N_\Q$ to define the non-Archimedean counterpart of $E_\eta$.

\begin{dfn}
The non-Archimedean counterpart of the modified Monge-Amp\`ere energy $E_\eta$ is defined to be  
\begin{equation*}
E_\eta^\NA(\cX, \cL) := \langle (\cX, \cL), 1+\eta \rangle.  
\end{equation*}
We introduce the modifed non-Archimedean energies as 
$D_\eta^\NA(\cX, \cL) := L^\NA(\cX, \cL) - E^\NA_\eta(\cX, \cL)$ and 
$J_\eta^\NA(\cX, \cL) := L_0^\NA(\cX, \cL) - E^\NA_\eta(\cX, \cL)$. 
\end{dfn}

Note that $J_\eta^\NA$ is not necessarily non-negative, just as $J_\eta$ was. 
We shall see that $T \subset \Aut(X, \eta)$ is enough to examine the positivity of $J_\eta^\NA$. 

\begin{prop}
If $m_X > 0$, then $J_\eta^\NA(\cX, \cL)\geq 0$ and the equality holds iff the $T$-equivariant $(\cX, \cL)$ is the trivial test configuration. 
\end{prop}

\begin{proof}
It is immediate from Theorem \ref{analytic formula for inner product of test configuration and 1-PS} and \ref{E in terms of action} that 
\begin{equation}\label{analytic J_eta}
J_\eta^\NA(\cX, \cL) = \sup_X \dot{\phi}^0 - \frac{1}{V}\int_X \dot{\phi}^0 (1+h_{\eta})\omega^n. 
\end{equation}
Since we may rescale the $\G_m$-action to have $\sup_X \dot{\phi}^0=0$, from the formula $m_X = \inf_X (1+h_\eta)>0$ implies $J_\eta^\NA(\cX, \cL) >0$, otherwise $\dot{\phi}^0$ is identically zero. 
By \cite{BHJ15} Theorem A, $\dot{\phi}^0 \equiv 0$ implies that $(\cX, \cL)$ is trivial. That is, the product configuration with the trivial action. 
\end{proof}

In terms of the associated concave function, we may write 
\begin{equation}
J_\eta^\NA(\cX, \cL)= 
\max_P G_{(\cX, \cL)} - \frac{1}{V} \int_P G_{(\cX, \cL)}G_{1+\eta} \DH_T. 
\end{equation}
In our definition of stability we assume $m_X > 0$. 
By Proposition \ref{m_X formula}, this additional assumption is very much easier to check than the positivity of $D_\eta^\NA$ for all test configurations. 

Let us return to a general $\mu \in N$ and take a $G$-equivariant trivialization so that the weights $\lambda_i$ and $\mu_i$ are assigned for the common vectors. 
We endow a new $\G_m$-action with the space $(\cX, \cL)$ such that the weights are given by $\lambda_i+\mu_i$. 
Since $T=C(G)$, it indeed gives a $G$-equivariant test configuration which we will denote by $(\cX_\mu, \cL_\mu)$. 
If $\Phi$ is the weak geodesic ray associated with $(\cX, \cL)$, it is easy to see that 
\begin{equation}
\phi_\mu^t(x) := \Phi(\lambda(e^{-t})\mu(e^{-t})(x, 1))  
\end{equation} 
gives the geodesic ray associated with $(\cX_\mu, \cL_\mu)$. 
The homogeneity naturally extends the definition to arbitrary $\mu \in N_\Q$. 
From Theorem \ref{analytic formula for inner product of test configuration and 1-PS}, we may further observe that $J^\NA(\cX_\mu, \cL_\mu)$ is continuous in $\mu \in N_\Q$. 
The next rationality lemma is a key to show the existence of the metric in Theorem A. 

\begin{lem}\label{rationality}
The functional $\mu \mapsto J_\eta^\NA(\cX_\mu, \cL_\mu)$ is rationally piecewise-linear convex function in $N_\R$. 
It is moreover strictly convex in $N_\R/\R$. 
Especially the infimum 
\begin{equation}
J^\NA_T(\cX, \cL):= \inf_{\mu \in N_\R} J^\NA(\cX_\mu, \cL_\mu)  
\end{equation} 
is attained by a rational $\mu$. 
\end{lem} 
\begin{proof}
The result was observed in \cite{His18}. 
Indeed by Proposition \ref{E in terms of action} we see that 
\begin{equation*}
J_\eta^\NA(\cX_\mu, \cL_\mu) 
=\max_i \frac{\lambda_i +\mu_i}{k}- \frac{1}{N_k}\sum_{i=1}^{N_k} (\lambda_i +\mu_i)(1+\eta_i).  
\end{equation*}
Thanks to Proposition \ref{E in terms of action} the first term is independent of $k$, as soon as $k\cL$ is globally generated. Note that the condition is independent of $\mu$.  
The second term is affine in $\mu$. 
Therefore, as the function in $\mu$, it is the maximum for finite number of affine functions. 
The function is obviously non-negative and proper in $N_\R/\R$. 
\end{proof} 

The notation $J_T$ and $J_T^\NA$ are consistent. 
We indeed have the slope formula which is the main ingredient in deriving the stability in Theorem A. 

\begin{thm}[\cite{His18}, Theorem B]\label{slope of J_T}
Let $(\cX, \cL)$ be a $T$-equivariant test configuration and $\phi^t$ be the associated weak geodesic ray. We have 
\begin{equation*}
J^\NA_T(\cX, \cL)=\lim_{t \to \infty} \frac{J_T (\phi^t)}{t}. 
\end{equation*} 
\end{thm} 

Notice that $\s \in T$ attaining the infimum of $J_T(\phi^t)$ depends on $t$. 
It is at least technically crucial to fix one torus in obtaining this sort of slope formulas. 
See \cite{His18}, Remark 1.8. 
Based on the results, we now arrive at the definition of the desired stability. 

\begin{dfn}\label{relative uniform D-stability}
A Fano manifold $X$ is uniformly relatively D-stable if $m_X >0$, $G=\Aut(X, \eta)$ is reductive, and there exists a constant $\e>0$ such that 
\begin{equation*}
D_\eta^\NA(\cX, \cL) \geq \e J_T^\NA(\cX_\mu, \cL_\mu) 
\end{equation*}
holds for any $G$-equivariant test configuration. 
We say that $X$ is relatively D-semistable if $D_\eta^\NA(\cX, \cL)\geq 0$ for any $G$-equivariant test configuration.   
\end{dfn}

\begin{rem}
In a recent preprint \cite{Yao19} it was proved that The obstruction $m_X>0$ about the automorphism group  automatically follows from the condition $D_\eta^\NA(\cX, \cL) \geq \e J_T^\NA(\cX_\mu, \cL_\mu)$. 
\end{rem} 

Our formulation endows the test configurations with large symmetry $G$.
This is considerably effective in checking the stability of specific Fano manifolds. 

\begin{exam}\label{example of G-invariant ideals}
As we have observed, there is no $G$-invariant ideal when $X=\P^2$. 
It simply implies that $\P^2$ is uniformly relatively D-stable. 
When $X$ is the one point blow-up of $\P^2$, we have the deformation to the normal cone $(\cX, \cL)$ for the exceptional divisor. 
We may check that $(\cX, \cL)$ dominates the product test configuration generated by $\eta$. 
It means that $D_\eta^\NA (\cX, \cL)=J_T^\NA(\cX, \cL)=0$. 
Indeed $X$ admits a Mabuchi soliton, and hence it is uniformly relatively D-stable, by the following general result. 
See \cite{Yao17} for investigation of the general toric Fano manifolds. 
\end{exam} 

In our framework explained so far, existence of the metric naturally implies the stability. 

\begin{thm}\label{metric implies stability}
If a Fano manifold admits a Mabuchi soliton, then it is uniformly relatively D-stable. 
\end{thm} 

\begin{proof}
By Theorem \ref{metric vs coercivity} we have the coercivity. 
As a consequence of Theorem \ref{slope of D} and Theorem \ref{slope of J_T}, the coercivity implies the stability. 
\end{proof}


\section{Variational approach and proof of the main theorem} 

Standing on the preparation of the last two sections we give a proof of Theorem A. 
After we organized the formulation, the argument is now a simple extension of the variational approach \cite{BBJ15}, to the relative and equivariant setting.  

\subsection{Convergence of weak geodesics}\label{convergence of weak geodesics}

Existence of the metric implies the stability, by Theorem \ref{metric implies stability}. 
Let us assume that a Fano manifold is uniformly relatively D-stable, in the sense of Definition \ref{relative uniform D-stability}. 
Since $J$ and $J_\eta$ are equivalent, one may use $J_\eta$ in replace of $J$. 
In view of the thermodynamical formalism, we have already observed in section \ref{thermodynamical formalism and modified K-energy} that the coercivity properties of the modified D and K-energy (see Definition \ref{modified K-energy}) are equivalent. 
We shall suppose that the modified K-energy is not coercive and lead the contradiction.  

The first step is to construct a weak geodesic ray in which direction the modified K-energy (and therefore D-energy) is not coercive. 
If the coercivity of Definition \ref{coercivity of D_g} fails, we have a sequence $\phi_j \in \cH(X, -K_X)^K$ and $\e_j \to 0$ $(j=1, 2, \dots)$ so that 
\begin{equation}\label{not coercive}
M_\eta(\phi_j) \leq \e_j J(\s^*\phi_j) -j 
\end{equation}
holds for {\em any} $\s \in T$. 
Since both sides are preserved by the constant rescaling $\phi \mapsto \phi+c$ we may take 
\begin{equation}\label{sup}
\sup_X (\phi_j-\phi_0) =0. 
\end{equation} 
We may moreover assume 
\begin{equation}\label{E diverges}
T_j:=-E(\phi_j) \to \infty,  
\end{equation} 
otherwise the uniform version of Skoda's integrability and the weak-compactness of the level set $\{ \phi \in \cE^1(X, -K_X)^K: E(\phi) \geq -C\} $ imply  
\begin{equation*}
D_\eta (\phi_j) \geq -\log C -E_\eta(\phi_j) 
\geq -\log C'. 
\end{equation*} 
Here we used again the comparison (\ref{E_g vs E}) of $E$ and $E_\eta$. 
Then (\ref{not coercive}) yields $J(\phi_j) \to \infty$, which contradicts to the assumption $E(\phi_j) \geq - C$ with (\ref{sup}). 

Let us take a weak geodesic $\phi_j^t$ $(0\leq t \leq -E(\phi_j))$ which joins $\phi_0$ to $\phi_j$. 
For the convergence of $\phi_j^t$ we need to control the relative entropy. 

\begin{thm}[\cite{BBEGZ16}, Theorem 2.17]\label{strong compactness}
The sublevel set 
\begin{equation*}
\bigg\{ \phi \in \cE^1(X, -K_X)^K: H(\MA(\phi)\vert \mu_0) \leq C,~ \sup_X (\phi-\phi_0)=0 \bigg\} 
\end{equation*} 
is compact in the $d_1$-topology. 
\end{thm} 

Since $g >0$ it is sufficient to control $H(\MA_g(\phi)\vert \mu_0)$. 
In view of the formula (\ref{Chen type formula}), the entropy bound is reduced to control $M_g$. 
Indeed we may control the last two terms in (\ref{Chen type formula}) by $E$, observing (\ref{E_g vs E}) and the elementary estimate 
\begin{equation}
(n+1)E(\phi) \leq \int_X (\phi-\phi_0) \MA(\phi) \leq E(\phi). 
\end{equation} 
The convexity of the modified K-energy now implies 
\begin{align}\label{convexity}
M_\eta (\phi_j^t) 
\leq \frac{t}{-E(\phi_j)} M_\eta (\phi_j) 
 \leq \frac{t}{-E(\phi_j)} (\e_j J (\s^*\phi_j)-j). 
\end{align} 
In particular for $\s=\id$ we obtain the bound of $M_\eta (\phi_j^t)$. 
It follows from Theorem \ref{strong compactness} that for each fixed $T$, $\phi_j^t$ $(0 \leq t \leq T)$ is contained in a compact subset with respect to the strong topology. 
The geodecity as well implies 
\begin{equation*}
d_1(\phi_j^t, \phi_j^s) =d_1(\phi_j^1, \phi_0) \abs{t-s}
\leq C(J(\phi_j^1)+1) \abs{t-s} 
\end{equation*}
for any $t, s \geq 0$. 
By Ascoli's theorem, passing through a subsequence if necessary, we conclude that $\phi_j^t$ strongly converges to $\phi^t$. 
It is immediate from $E(\phi_j^t)=-t$ that $E(\phi^t)= -t$.

\subsection{Demailly type approximation}\label{Demailly type approximation}

The second step is to approximate $\phi^t$ constructed in the above by a sequence of test configurations. 
It is the non-Archimedean analogue of Demailly's approximation theorem for plurisubharmonic functions. 
Given $\phi^t$, the relation (\ref{compatibility of ray}) gives the singular $K$-invariant metric $\Phi$ on $L_{\A^1}$, defined over $X_{\A^1\setminus \{0\}}=\C^* \times X$. 
Since $\sup_X (\phi^t -\phi_0)=0$, the plurisubharmonicity uniquely extends $\Phi$ to $\A^1$. 
Now for a sufficiently large $m \in \N$ we take the multiplier ideal sheaf $\cJ(m\Phi)$ and the normalized blow-up $\rho_m\colon  \cX^{(m)} \to \A^1$, endowed with the exceptional divisor $E_m$ and the line bundle 
\begin{equation}
\cL^{(m)} := \rho_m^*L_{\A^1} - \frac{1}{m+m_0} E_m.  
\end{equation} 
We may show that $\cL^{(m)}$ is relatively semiample line bundle. See \cite{BBJ15}, Lemma $5.6$ for the proof. 
We may check that the test configuration $(\cX^{(m)}, \cL^{(m)})$ inherits the equivariant $G$-action, since $\cJ(m\Phi)$ is $G$-invariant. 
Note that the central fiber $\cX_0$ is the union of the strict transform $E_0$ of $X \times \{ 0\}$ and the exceptional divisor $E_m$.  
The $\G_m$-action of $(\cX^{(m)}, \cL^{(m)})$ is trivial on $E_0$ so that it commutes with the $G$-action. 

\begin{thm}[\cite{BBJ15}, Theorem $5.4$ and $6.4$ for the $T=\{\id\}$ case]\label{BBJ}
For the above test configurations constructed from $\phi^t$, we have 
\begin{align*}
&E_\eta^\NA(\cX^{(m)}, \cL^{(m)}) \geq \lim_{t \to \infty} \frac{E_\eta(\phi^t)}{t}, \\
&\lim_{m \to \infty} L^\NA(\cX^{(m)}, \cL^{(m)}) = \lim_{t \to \infty} \frac{L(\phi^t)}{t}.  
\end{align*}
\end{thm} 

We need the modified $E_\eta$ in the above, however, the proof is the same as \cite{BBJ15}. 
Indeed, using Demailly's approximation theorem locally, we have the estimate 
\begin{equation}
\Phi_m \geq \Phi -C_{m, r}
\end{equation}
on the shrunken area $\B(0, r)\times X$. 
The constant $C_{m, r}$ is necessarily independent of $t$. 
Since the modified Monge-Amp\`ere energy is monotone, we apply Corollary \ref{slope of modified E} to obtain 
\begin{align*}
E_\eta^\NA(\cX^{(m)}, \cL^{(m)})
&= \lim_{t\to \infty} \frac{E_\eta(\phi^t_m)}{t} \\
&\geq \lim_{t\to \infty} \frac{E_\eta(\phi^t-C_{m, r})}{t}
=\lim_{t \to \infty}\frac{E_\eta(\phi^t)}{t}.
\end{align*}
The key point in the above is the Ohsawa-Takegoshi $L^2$-extension theorem \cite{OT87} used in Demailly's approximation. 

Let us next consider the upper bound of $D(\cX_m, \cL_m)$. 
With the convexity of the D-energy, the assumption (\ref{not coercive}) for $\s =\id$ immediately implies  
\begin{equation}
\lim_{t \to \infty} \frac{D_\eta (\phi^t)}{t} \leq  0. 
\end{equation} 
It then follows 
\begin{equation}\label{upper bound of L}
\begin{aligned}
\lim_{m \to \infty} L^\NA(\cX^{(m)}, \cL^{(m)})
&= \lim_{t \to \infty} \frac{L(\phi^t)}{t} \\
&=\lim_{t \to \infty} \frac{D_\eta(\phi^t) +E_\eta (\phi^t)}{t}
\leq \hat{E_\eta} (\Phi) =\langle \Phi, \eta \rangle. 
\end{aligned}
\end{equation} 
This is satisfactory for our purpose. 

We aim for the lower bound estimate of $L^\NA$ to get the contradiction. 
From now on we follow the strategy of \cite{Li19} modifying the original idea of \cite{BBJ15}.  
Let us take some $\mu_m \in \N_\Q$ by Lemma \ref{rationality} so that 
\begin{equation*}
J^\NA(\cX^{(m)}_{\mu_m}, \cL^{(m)}_{\mu_m}) 
= J^\NA_{T}(\cX^{(m)}, \cL^{(m)}) := \inf_{\mu \in N_\R} J^\NA(\cX^{(m)}. \cL^{(m)}) 
\end{equation*} 
holds. 
To subtract the convergent subsequence of $\mu_m$, let us serve another simple proof of the boundedness lemma in \cite{Li19}.  

\begin{lem}[\cite{Li19}, discussion in the section 5.4]\label{boundedness} 
The achievements $\mu_m$ is bounded in the vector space $N_\R$. 
Specifically, taking a norm of $N_\R$ we have a constant independent of $m$ such that 
\begin{equation*}
\abs{\mu_m} \leq C. 
\end{equation*} 
\end{lem} 

\begin{proof}
It is sufficient to bound $J^\NA (-\mu_m)$ or equivalent I-functional 
\begin{equation*}
\lim_{t \to \infty} \frac{I(\mu_m(e^{t})^*\psi^0, \psi^0)}{t}
\end{equation*}
for some fixed $\psi^0$. See (\ref{I-functional}) for the definition. 
We may write $\mu_m(e^{t})^*\psi^0=\psi^0_{(-\mu_m)}$ if regard $\psi^0$ as the constant ray. 
Let $\psi^t$ be the associated ray with $(\cX^{(m)}, \cL^{(m)}$).  
Our trick using the pseudo-triangle inequality (\ref{pseudo-triangle inequality}) is 
\begin{align*}
c_n I(\psi^0_{(-\mu_m)}, \psi^0)
&\leq I(\psi^0_{(-\mu_m)}, \psi^t) + I(\psi^t, \psi^0) \\
&= I(\psi^0, \psi^t_{\mu_m}) +I(\psi^t, \psi^0). 
\end{align*} 
We may control the first term by $J^\NA_{T}(\cX^{(m)}, \cL^{(m)})$ and the second term by $J^\NA(\cX^{(m)}, \cL^{(m)})$. 
These two are bounded because 
\begin{equation*}
J^\NA(\cX^{(m)}, \cL^{(m)})
= -E^\NA(\cX^{(m)}, \cL^{(m)}) \leq \frac{-E(\phi^t)}{t}=1
\end{equation*}
and obviously $J^\NA_{T}(\cX^{(m)}, \cL^{(m)}) \leq J^\NA(\cX^{(m)}, \cL^{(m)})$. 
Notice that the similar argument was also applied to the proof of the slope formula. 
\end{proof} 

By the above lemma, taking a subsequence if necessary we may assume $\mu_m$ converges to some $\mu \in N_\R$. 
The next lemma crucial in the proof shows that the twisted ray $\phi_\mu^t$ is non-trivial. 
This is the point we use the assumption (\ref{not coercive}) for arbitrary $\s \in T$. 

\begin{lem}[\cite{Li19}, Corollary 5.3]\label{non-triviality} 
The radial J-energy of the twisted ray $\phi_\mu^t$ is strictly positive. Namely 
\begin{equation*}
\hat{J}_\eta(\Phi_\mu) >0. 
\end{equation*} 
\end{lem} 
\begin{proof}
Here we follow the discussion of \cite{Li19}. 
In view of (\ref{J_g vs J}), it is sufficient to prove $\hat{J}(\Phi_\mu) >0$. 
We observe from (\ref{Chen type formula}) that 
\begin{align*}
M(\phi_j) =M(\s^*\phi_j) \geq C-nJ(\s^*\phi_j)
\end{align*} 
holds. The assumption (\ref{not coercive}) then yields 
\begin{align*}
J(\s^*\phi_j) \geq \frac{j+C}{n+\e_j} \to +\infty. 
\end{align*}
It implies $J(\phi_\mu^t) \to +\infty$. 
Since $J$ has linear growth along geodesics, 
we obtain $\hat{J}(\Phi_\mu) >0$.  
\end{proof}

The above two lemmas furnish the proof of the main theorem. 
Let us decompose the L-functional as 
\begin{align*}
L^\NA (\cX^{(m)}, \cL^{(m)}) 
= D_\eta^\NA(\cX^{(m)}, \cL^{(m)}) +E_\eta^\NA (\cX^{(m)}, \cL^{(m)}). 
\end{align*}
By the uniform stability (with in mind of (\ref{J_g vs J}) again) there exists some $\e'>0$ such that the right-hand side is not less than 
\begin{align*} 
&\e' J_\eta^\NA(\cX^{(m)}_{\mu_m}, \cL^{(m)}_{\mu_m}) + E_\eta^\NA (\cX^{(m)}, \cL^{(m)}) \\
&= \e' L_0^\NA (\cX^{(m)}_{\mu_m}, \cL^{(m)}_{\mu_m}) +(1-\e')E_\eta^\NA (\cX^{(m)}_{\mu_m}, \cL^{(m)}_{\mu_m}) - \langle \mu_m, 1+\eta\rangle . 
\end{align*}
As we may assume $1-\e' >0$, the second term is now controlled by Theorem \ref{BBJ}. 
Since the functional $L_0(\phi)$ is monotone in $\phi$, we may conclude 
\begin{align*} 
\lim_{m \to \infty} L^\NA (\cX^{(m)}, \cL^{(m)}) 
 & \geq \e' \hat{L_0}(\Phi_\mu) +(1-\e')\hat{E_\eta}(\Phi_\mu) - \langle \mu, 1+\eta\rangle  \\
 &= \e' \hat{J_\eta}(\Phi_\mu) + \hat{E_\eta}(\Phi). 
\end{align*}
By Lemma \ref{non-triviality} this is strictly greater than $\langle \Phi, \eta \rangle$ and it 
contradicts to \ref{upper bound of L}. 


\end{document}